\documentclass[11pt]{amsart}

\usepackage{amssymb,amsmath}

\usepackage{amsthm,enumerate}
\usepackage{latexsym}
\usepackage{amsbsy}
\usepackage{dsfont,txfonts}
\usepackage{yhmath}
\usepackage{bbold}

\theoremstyle{plain}
\newtheorem{lemma}{Lemma}
\numberwithin{lemma}{section}
\newtheorem{corollary}{Corollary}
\numberwithin{corollary}{section}
\newtheorem{definition}{Definition}
\numberwithin{definition}{section}
\newtheorem{theorem}{Theorem}
\numberwithin{theorem}{section}

\newcommand\orb{\textrm{orb}}
\newcommand\degr{\textrm{  deg}}
 \newcommand\supp{\textrm{  supp}}
 \newcommand\C{{\mathcal C}}
\def\P{{\mathcal P}}
 \newcommand\T{{ T}}
\newcommand\X{{\mathcal X}}
 \newcommand\R{{\mathbb R}}
 \newcommand\1{{\mathbb{1}}}

\title[Transport-entropy inequalities on locally acting  groups]
{Transport-entropy inequalities\\ on locally acting  groups of permutations}

\author[P.-M. Samson]{ Paul-Marie Samson}

\date{\today}

\thanks{Supported by the grants ANR 2011 BS01 007 01, ANR 10 LABX-58}

\address{Samson P.-M., Universit\'e Paris-Est, Laboratoire d'Analyse et de Math\'ematiques Appliqu\'ees (UMR 8050), UPEM, UPEC, CNRS, F-77454, Marne-la-Vall\'ee, France}
\email{paul-marie.samson@u-pem.fr}

\begin{document}
\maketitle
\begin{abstract} Following Talagrand's concentration results for permutations picked uniformly at random from a symmetric group \cite{Tal95}, Luczak and McDiarmid have generalized it to  
more general groups $G$ of permutations which act suitably `locally'.  Here we extend  their results by setting transport-entropy inequalities on these permutations groups. Talagrand and Luczak-Mc-Diarmid concentration properties are consequences of these inequalities. The results are  also generalised   to a larger class of measures including Ewens distributions of arbitrary parameter $\theta$ on the symmetric group.  By projection, we derive transport-entropy inequalities for the uniform law on the slice of the discrete hypercube  and more generally for the multinomial law.  
These  results are new  examples, in discrete setting, of weak transport-entropy  inequalities introduced in \cite{GRST16}, that contribute to a better understanding of the concentration properties of measures on  permutations groups.   One typical application is deviation bounds for the so-called configuration functions, such as the number of cycles of given lenght  in the cycle decomposition of a random permutation.
\end{abstract}


\date{\today}

\section{Introduction}

Let $S_n$ denote the symmetric group of permutations acting on a set $\Omega$ of cardinality $n$, and   $\mu_o$ denote the uniform law on  $S_n$, 
$\mu_o(\sigma):=\frac{1}{n!}, \sigma\in S_n.$ A seminal concentration result on $S_n$ obtained by  Maurey   is the following. 

 \begin{theorem}\cite{Mau79}\label{permmau} Let  $d_H$ be the  Hamming distance on the symmetric group, for  all $\sigma,\tau\in S_n$, 
\[d_H(\sigma,\tau):=\sum_{i\in \Omega} \1_{\sigma(i)\neq \tau(i)}.\]
Then for any subset  $A\subset S_n$ such that  $\mu_o(A)\geq 1/2$, and for all   $t\geq 0$, one has  
\[\mu_o(A_t) \geq 1-2e^{-\frac{t^2}{64n}},\]
where $A_t:=\{y\in S_n, d_H(x,A)\leq t\}$.
\end{theorem}
Milman and Schechtman \cite{MS86} generalized this result  to some groups whose distance is  invariant by translation. For example,
in the above result  we may  replace (up to constants) the Hamming distance by the  \textit{ transposition distance} $d_T(\sigma,\tau)$ that corresponds to the minimal number of transpositions $ t_1,..., t_k$ such that $\sigma t_1\cdots t_k=\tau$. The distances $d_T$ and $d_H$ are comparable,  \[\frac12 d_H(\sigma,\tau)\leq d_T(\sigma,\tau)\leq d_H(\sigma,\tau)-1,  \qquad \forall \sigma\neq\tau.\]
(We refer to \cite{BHT06} for comments  about these comparison  inequalities).
 
 Let us also observe that Theorem \ref{permmau} can be also recover  from the transportation cost inequality approach of Theorem 1 of \cite{Mar03}.
 
A few years later, a stronger    concentration property  in terms of dependence in the parameter  $n$,  has been settled   by Talagrand using the so-called ``convex-hull'' method \cite{Tal95} (see also \cite{Led01}). 
 This property implies Maurey's result  with a slightly worse constant. 
 Let us recall some notations  from \cite{Tal95}. For each $A \subset S_n$ and $\sigma \in S_n$, let $V(\sigma,A)\subset \R^\Omega$ be the set of vectors $z =(z_j)_{j\in\Omega} \in \R^\Omega$ with $z_j := \1_{\sigma(j)\neq y(j)}$ for $y \in A.$ Let $\mathrm{conv}(V(\sigma,A))$ denote the convex hull of $V(\sigma,A)$ in $\R^\Omega$, 
 \[V(\sigma,A):=\left\{  x =(x_j)_{j\in\Omega}, \exists p\in \mathcal{P}(A),\forall j\in \Omega, x_j=\int  \1_{\sigma(j)\neq y(j)} dp(y)\right\},\]
 where $\P(A)$ denotes the set of probability measures on $A$. Talagrand introduced the quantity 
\[
f(\sigma,A) := \inf\{ \|x\|^2_2 ;  x\in \mathrm{conv}(V(\sigma,A))\}.
\]
with $\|x\|^2_2:=\sum_{i\in \Omega} x_i^2$, that measures   the distance from $\sigma$ to the subset $A$. 

\begin{theorem}\label{tala95}\cite{Tal95} For any subset  $A\subset S_n$, 
\[\int_{S_n} e^{f(\sigma,A)/16} d\mu_o(\sigma)\leq \frac1{\mu_o(A)}. \]
\end{theorem}
Maurey's concentration result easily follows by observing that 
\[f(\sigma,A)\geq \frac1n \left(\inf\left\{ \sum_{i\in \Omega} x_i ;  x\in \mathrm{conv}(V(\sigma,A))\right\}\right)^2=\frac1n d_H^2(\sigma,A)\]
and applying Tchebychev inequality with usual optimization arguments.

Talagrand's result  has been first extended to the uniform probability measure on  product of symmetric groups by McDiarmid \cite{McD02},  and then further by  Luczak and McDiarmid  to  cover more general permutation groups which act suitably ``locally'' \cite{LD03}.

For any finite subset $A$, let $\#A$ denote the \textit{cardinality} of $A$. 
For any  $\sigma\in S_n$, the  \textit{ support} of $\sigma$, denoted by $\supp(\sigma)$, is the set $\{i\in\Omega, \sigma(i)\neq i\}$ and the  \textit{ degree} of $\sigma$, denoted by $\degr(\sigma)$, is the cardinality of $\supp(\sigma)$, $\degr(\sigma):=\#\,\supp(\sigma) $. 

By definition, according to  \cite{LD03}, a group of permutations $G$ is  \textit{ $\ell$-local}, $\ell\in\{2,\ldots,n\}$, if for any $\sigma\in G$ and  any $i,j\in \Omega$ with $\sigma(i)=j$, there exists $\tau\in G$ such that $\supp(\tau)\subset\supp(\sigma)$, $\degr(\tau)\leq \ell$ and   $\tau(i)=j$.

The  \textit{ orbit} of an element $j\in \Omega$, denoted by $\orb(j)$, is the set of elements in $\Omega$ connected to $j$ by a permutation  of $G$,
\[  \orb(j):=\big\{ \sigma(j), \sigma\in G\big\}. \]
The set of orbits provides a partition of $G$.

As explained in \cite{LD03}, any $2$-local group is a direct product of symmetric groups on its orbits, the alternating group (consisting of even permutations) is $3$-local, and  any $3$-local group is a direct product of symmetric or alternating groups on its orbits.

In the present  paper, the concentration result  by Luczak-McDiarmid  and Talagrand is a consequence of a weak transport-entropy inequality satisfied by the uniform law on $G$, $\mu_o$. We also prove weaker types of transport entropy inequalities. Moreover we extend the results to a larger class of probability measures on $G$, denoted by $\mathcal M$. 

For a better comprehension of the class of  measures $\mathcal M$, let us first consider the  case of the symmetric group $S_n$ on $[n]:=\{1,\ldots,n\}$. 
Let $(i,j)$ denote the transposition in $S_n$ that exchanges the elements $i$ and $j$ in $[n]$. 
It follows by induction   that the map 
\[\begin{array}{cccl} & \{1,2\}\times\{1,2,3\}\times \cdots \times \{1,\ldots , n\}& \to &S_n\\
U: &i_2,i_3, \dots ,i_n&\mapsto & (i_2,2)(i_3,3)\cdots (i_n,n), \end{array}\]
is one to one. 

The set of measures  $\mathcal M$ consists  of probability measures on $S_n$ which are  pushed forward by the map $U$ of  product probability measures on  $\{1,2\}\times\{1,2,3\}\times \cdots \times \{1,\ldots , n\}$, 
\begin{eqnarray}\label{defM}
\mathcal{M}:=\big\{U\#\hat \nu, \hat \nu=\hat \nu_2\otimes\cdots \otimes \hat \nu_n \mbox{ with }  \hat \nu_j\in\P([j]),\;\forall j\in \{2,\ldots, n\}\big\},
\end{eqnarray} 
where by definition $U\#\hat \nu(C)=\hat \nu(U^{-1}(C))$ for any subset $C$ in $S_n$.

The uniform measure $\mu_o$ on $S_n$ belongs to the set $\mathcal M$ since $\mu_o=U\#\hat \mu$ with $\hat \mu=\hat \mu_2\otimes\cdots \otimes \hat \mu_n$, where for each $i$, $\hat \mu_i$ denotes the uniform law on $[i]$. 

The Ewens distribution of parameter $\theta>0$, denoted by $\mu^\theta$, is also an example of measure of $\mathcal M$. Indeed, it is well known (see \cite[Chapter 5]{ABT03}, \cite{JKB97}) that $\mu^\theta=U\#\hat \mu^\theta$ with $\hat \mu^\theta=\hat \mu^\theta_2\otimes\cdots \otimes \hat \mu^\theta_n$, where for any $j\in\{2,\ldots,n\}$, the measure $\hat\mu^\theta_j\in {\mathcal P}([j])$ is given by
\[\hat\mu^\theta_j(j)=\frac{\theta}{\theta+j-1},\quad, \hat\mu^\theta_j(1)=\cdots= \hat \mu^\theta_j(j-1) =\frac{1}{\theta+j-1}.\]

This definition provides an easy algorithm  for simulating  a random permutation with law $\mu^\theta$. This procedure is known as  \textit{ a Chinese restaurant process} (see \cite[Chapter 2]{ABT03}, \cite{Pit02}).

Let us  observe that the uniform distribution $\mu_o$ corresponds to  the Ewens distribution with parameter 1, $\mu^1$.

The Ewens distribution is also given by the following expression (see \cite[Chapter 5]{ABT03}), 
\begin{eqnarray}\label{ewens}
\mu^\theta(\sigma):=\frac {\theta^{|\sigma|}}{\theta^{(n)}}, \quad \sigma\in S_n,
\end{eqnarray}
where 
 $|\sigma|$ denotes the number of cycles in the cycle decomposition of $\sigma$ and 
$\theta^{(n)}$  is the Pochhammer symbol defined by \[ \theta^{(n)}:=\frac{\Gamma(\theta+n)}{\Gamma(\theta)},\qquad\mbox{ with } \quad{\Gamma}(\theta):=\int_0^{+\infty} s^{\theta-1} e^{-s} ds.\]

Let us now construct the class of measures $\mathcal M$ for any  group $G$ of permutations. 
To clarify the notations, the elements of $\Omega$ are labelled with  integers, $\Omega=[n]$. Let $G_n:=G$ and for any $j\in [n-1]$, let $G_j$ denotes the subgroup of $G$ defined by 
\[ G_j:=\left\{\sigma\in G, \sigma(j+1)=j+1,\ldots,\sigma(n)=n\right\}, \]
We denote by  $O_j$ the orbit of $j$ in $G_j$,
\[O_j:=\left\{\sigma(j),\sigma \in G_j\right\}.\]
Let us observe that $\{j\}\subset O_j\subset [j]$.

\begin{definition} Let $G$ be a group of permutations. A family $\mathcal {T}=(t_{i_j,j})$ of permutations  of $G$, indexed by $j\in \{2,\ldots, n\}$ and $i_j\in O_j$, is  called  ``$\ell$-local base of $G$''
if for every $j\in \{2,\ldots, n\}$, $t_{j,j}:=id$, for every $i_j\neq j$, $t_{i_j,j}\in G_j$ and 
\[t_{i_j,j}(i_j)=j, \quad \mbox{and}\quad  \degr(t_{i_jj})\leq \ell.\]
\end{definition}
\begin{lemma}\label{map} Let $\mathcal {T}=(t_{i_j,j})$ be a $\ell$-local base of a group of permutations $G$. Then the map 
\begin{eqnarray}\label{TT}\begin{array}{cccl} & O_2\times O_3\times \cdots \times O_n& \to &G\\
U_{\mathcal T}: &i_2,i_3, \dots ,i_n&\mapsto & t_{i_2,2} t_{i_3,3} \cdots t_{i_n,n}, \end{array}
\end{eqnarray}
is one to one. 
\end{lemma}

\begin{lemma}\label{lembase} Any  $\ell$-local group of permutations admits  a ``$\ell$-local base''.
\end{lemma}
For completeness, a proof of these two lemmas is given in the Appendix. 

As a consequence of these lemmas, if $G$ is  a $\ell$-local group, then there exists a $\ell$-local base $\mathcal T$, such that the uniform probability measure $\mu_o$ satisfies  
$\mu_o=U_{\mathcal T}\#\hat \mu$, with $\hat \mu=\hat \mu_2\otimes\cdots \otimes \hat \mu_n$, where for each $j$, $\hat \mu_j$ is the uniform law on $O_j$.

As for the symmetric group, given a $\ell$-local base ${\mathcal T}$ of  a group $G$, the class of measures  $\mathcal M=\mathcal M_{\mathcal T}$ on $G$ is made up of all probability measures on $G$ which are pushed forward of product probability measures on $ O_2\times O_3\times \cdots \times O_n$ by the map $U_{\mathcal T}$ defined by \eqref{TT},
\begin{eqnarray}\label{setmeas}
\mathcal{M}_{\mathcal T}:=\big\{U_{\mathcal T}\#\hat \nu,  \hat \nu=\hat \nu_2\otimes\cdots \otimes \hat \nu_n \mbox{ with }  \hat \nu_j\in\P(O_j),\;\forall j\in \{2,\ldots , n\}\big\}.
\end{eqnarray}
 As explained above, if $G$ is a $\ell$-local group, the class $\mathcal M_{\mathcal T}$ contains the uniform law $\mu_o$ on $G$ for a well choosen $\ell$-local base ${\mathcal T}$.

In this paper, the concentration results are derived from weak transport-entropy inequalities, involving the relative entropy $H(\nu|\mu)$ between two probability measures $\mu,\nu$ on $G$ given by 
\[ H(\nu|\mu):=\int \log\left(\frac{d\nu}{d\mu}\right) d\nu,\]
if $\nu$ is absolutely continuous with respect to $\mu$ and $H(\nu|\mu):=+\infty$ otherwise.

The terminology ``weak transport-entropy'' introduced in \cite{GRST16},  encompass many kinds of transport-entropy  inequalities from the well-known Talagrand's transport inequality  satisfied by the standard Gaussian measure on $\R^n$ \cite{Tal96c}, to the usual  Csiz\'ar-Kullback-Pinsker  inequality \cite{Pin64,Csi67,Kul67} that holds for any (reference) probability measure $\mu$ on a Polish metric space $\X$, namely
\begin{eqnarray}\label{CKP}
 \|\mu-\nu\|_{TV}^2\leq  2 \,H(\nu|\mu),\qquad \forall \nu\in\P(\X).
 \end{eqnarray}
 where  $ \|\mu-\nu\|_{TV}$ denotes the total variation distance between $\mu$ and $\nu$, 
 \[\|\mu-\nu\|_{TV}:=2\sup_{A}|\mu(A)-\nu(A)|.\]
 Above, the supremum runs over all measurable subset $A$ of $\X$.
 We refer to  the survey \cite{Sth16, Sam16}  for other examples of  weak transport-entropy inequalities and their connections  with  the concentration of measure  principle.  
  
  The next theorem is one of the main result of this paper. It presents new weak transport inequalities  for the uniform measure on $G$ or any measure  in the class $\mathcal M_{\mathcal T}$, that   recover  the  
concentration results of 
Theorems \ref{permmau} and \ref{tala95}.

We also denote by $d_H$ the Hamming distance on $G$: for any $\sigma,\tau\in G$,
\[d_H(\sigma,\tau):= \degr(\sigma \tau^{-1})=\sum_{i=1}^n \1_{\sigma(i)\neq \tau(i)},\]
 and the distance $d_T(\sigma,\tau)$ is defined as the minimal number of elements of $G$, $ t_1,..., t_k$, with degree less than $\ell$, such that $\sigma t_1\cdots t_k=\tau$.   

For any measures $\nu_1,\nu_2\in \P(G)$, the set $\Pi(\nu_1,\nu_2)$ denotes the set of all probability measures on $G\times G$ with first marginal $\nu_1$ and second marginal $\nu_2$.  The Wasserstein distance between $\nu_1$ and $\nu_2$, according to the distance $d=d_H$ or $d=d_T$, is given by 
\[W_1(\nu_1,\nu_2):=\inf_{\pi\in \Pi(\nu_1,\nu_2)}\iint d(\sigma,\tau)\, d\pi(\sigma,\tau).\] 
We also consider two other optimal weak transport costs,  ${\widetilde T}_2(\nu_2|\nu_1)$ and $\wideparen \T_2(\nu_2|\nu_1)$ defined by 
 \begin{eqnarray}\label{defTtilde}
  {\widetilde T}_2(\nu_2|\nu_1):=\inf_{\pi\in \Pi(\nu_1,\nu_2)}\int \left(\int d(\sigma,\tau) \,dp_\sigma(\tau)\right)^2 d\nu_1(\sigma), 
  \end{eqnarray}
  and 
 \[\wideparen T_2(\nu_2|\nu_1):=\inf_{\pi\in \Pi(\nu_1,\nu_2)}\int \sum_{i=1}^n \left(\int \1_{\sigma(i)\neq \tau(i)} \,dp_\sigma(\tau)\right)^2 d\nu_1(\sigma),\]
 where $p_\sigma$ represents any probability measure such that   $\pi(\sigma,\tau)=\nu_1(\sigma)p_\sigma(\tau)$ for all $\sigma,\tau\in G$.
 By Jensen's inequality, these weak transport costs are comparable, namely 
 \[W_1^2(\nu_1,\nu_2)\leq {\widetilde T}_2(\nu_2|\nu_1)\leq  n \wideparen T_2(\nu_2|\nu_1),\]
 where the last inequality only   holds for $d=d_H$.

By definition a subgroup $G$ of $S_n$ is \textit{normal} if for any $t\in S_n$,$t^{-1}G t=G$.

In the next theorem the constant $K_n$ is the cardinality of the set $\big\{  j\in\{2,\ldots,n\}, O_j\neq \{j\}\big\}$. It follows that $0\leq K_n\leq (n-1)$ and $K_n=0$ if and only if $G=\{id\}$. 
\begin{theorem}\label{transportsymetrique} Let $G$ be a  group of permutations  with  $\ell$-local base $\mathcal T$. Let $\mu\in \P(G)$ be a  measure  of the set $\mathcal M_{\mathcal T}$  defined by \eqref{setmeas}.
\begin{enumerate} 
\item[(a)] For all probability measures  $\nu_1$ and $\nu_2$ on $G$, one has
\begin{eqnarray}\label{TW1}
\frac2{c(\ell)^2} {W}_1^2(\nu_1,\nu_2)\leq K_n\left(\sqrt{H(\nu_1|\mu)}+\sqrt{H(\nu_2|\mu)}\right)^2,
\end{eqnarray}
and
\begin{eqnarray}\label{Ttilde}\frac1{2c(\ell)^2} {\widetilde T}_2(\nu_2|\nu_1)\leq K_n\left(\sqrt{H(\nu_1|\mu)}+\sqrt{H(\nu_2|\mu)}\right)^2,
\end{eqnarray}
where 
\[c(\ell):=\begin{cases}\min(2\ell-1,n) & \mbox{ if } d=d_H,\\
2& \mbox{ if } d=d_T. \end{cases}\]
When $\mu=\mu_o$ is the uniform law of a $\ell$-local group $G$,  inequalities \eqref{TW1} and \eqref{Ttilde}  hold with 
\[c(\ell):=\begin{cases}\ell & \mbox{ if } d=d_H,\\
1& \mbox{ if } d=d_T. \end{cases}\]

\item[(b)] 
\begin{itemize}

\item 
Assume that $\mu=\mu_o$ is the uniform law of a $\ell$-local group $G$.  
Then, for all probability measures  $\nu_1$ and $\nu_2$ on $G$,
\begin{equation}\label{Tparen}
\frac1{2c(\ell)^2}  \wideparen \T_2(\nu_2|\nu_1)\leq \left( \sqrt{H(\nu_1|\mu)} +\sqrt{  H(\nu_2|\mu)}\right)^2,
\end{equation}
with $c(\ell)^2=2(\ell-1)^2+2$.
\item Assume that $G$ is a normal subgroup of $S_n$, 
and  that  $\mu$ satisfies  for all $\sigma\in G,t\in S_n$
\begin{eqnarray}\label{hypoinv}
\mu(\sigma)=\mu(\sigma^{-1})\quad\mbox{ and }\quad \mu(\sigma)=\mu(t^{-1}\sigma t). 
\end{eqnarray}
Then, the inequality \eqref{Tparen} holds 
 with  $c(\ell)^2=8(\ell-1)^2+2$. 
 \end{itemize}
 \end{enumerate}
\end{theorem}

The proofs of these results, given  in the next section, are  inspired by  Talagrand seminal work on $S_n$ \cite{Tal95},  and Luczak-McDiarmid extension to $\ell$-local groups \cite{LD03}. 

 \textbf{ Comments :}
\begin{itemize}
\item If $G=S_n$ and the class of measure $\mathcal M$ is given by \eqref{defM}, the Ewens distribution $\mu^\theta$ introduced before, is an interesting example of measure in $\mathcal M$, satisfying  condition \eqref{hypoinv}. This simply follows from its expression given by \eqref{ewens}, since for any $\sigma,t\in S_n$, $|\sigma^{-1}|=|\sigma|$ and $|t^{-1}\sigma t|=|\sigma|$. 

An open question is to generalize the above transport-entropy  inequalities to the  \textit{ generalized Ewens distribution} (see the definition in \cite{MNZ12,HNNZ13}). This measure  no longer belongs  to  the class of measure $\mathcal M$. In other words, no Chinese restaurant process are known for simulating the generalized Ewens distribution.


\item From the triangular inequality satisfied by the Wasserstein distance $W_1$, the transport-entropy  inequality \eqref{TW1} is clearly equivalent to the following transport-entropy inequality, for all probability measure $\nu $ on $G$, 
\begin{eqnarray*}
\frac2{c(\ell)^2} {W}_1^2(\nu,\mu)\leq K_n\, {H(\nu|\mu)}.
\end{eqnarray*}
Here is  a  popular dual formulation of this transport-entropy  inequality:  for all $1$-Lipschitz functions  $\varphi:G\to\R$ (with respect to  the  distance $d$),
\begin{eqnarray}\label{Hoeff}
\int e^\varphi d\mu\leq e^{\int \varphi\, d\mu+K_n c(\ell)^2  t^2/8},\qquad\forall t\geq 0.
\end{eqnarray}
 For the uniform measure on $S_n$, $K_n=n-1$ and this property is widely commented in  \cite{BHT06}; it  is also a  consequence of Hoeffding inequalities for bounded martingales (see page 18 of \cite{Hoe63}). The  concentration result derived from item (a)  are  of the same nature as the one obtained by the ``bounded differences approach'' in \cite{Mau79, McD89, McD02,LD03,BDR15}.

\item Similarly, by Proposition 4.5 and Theorem 2.7 of \cite{GRST16} and using the identity \[\left(\sqrt u+\sqrt v\right)^2 =\inf_{\alpha\in (0,1)}\left\{\frac u\alpha+\frac v{1-\alpha}\right\},\] we may easily show that the weak transport-entropy inequality \eqref{Ttilde} is equivalent to the following  dual property: 
 for any real function 
$\varphi$ on $G$ and for any 
  $0< \alpha<1$, 
\begin{eqnarray}\label{symtilde}
\left(\int e^ {\alpha  \widetilde Q_{K_n}\varphi} d\mu\right)^{1/\alpha}\left(\int e^{-(1-\alpha)\varphi} d\mu\right)^{1/(1-\alpha)}\leq 1,
\end{eqnarray}
where  the infimum-convolution operator $\widetilde Q_t \varphi$, $t\geq 0$,  is defined by 
\[ \qquad \widetilde Q_t\varphi(\sigma):=\inf_{p\in \P(G)} \left\{  \int \varphi \,dp  +\frac 1{2c^2(\ell)t}  \left(\int d(\sigma,y) \,dp(y)\right)^2 \right\}, \quad \sigma\in G.\]

Moreover, let us observe that following our proof of \eqref{symtilde} in the next section,  for each $\alpha\in (0,1)$  the inequality \eqref{symtilde} can be improved by replacing the square cost function by  the convex cost $c_\alpha(u)\geq u^2/2, u\geq 0$ given in  Lemma \ref{completegraph}. More precisely,  \eqref{symtilde} holds replacing $\widetilde Q_{K_n}\varphi$ by $\widetilde Q^\alpha_{K_n}\varphi$ defined by 
\[\qquad  \widetilde Q^\alpha_t\varphi(\sigma):=\inf_{p\in \P(S_n)} \left\{  \int \varphi \,dp  +t  c_\alpha\left(\frac1{c(\ell)t}\int d(\sigma,y) \,dp(y)\right)^2 \right\},   \]
for any $\sigma\in G,t>0$.
\item  Proposition 4.5 and Theorem 9.5 of \cite{GRST16} also provide a  dual formulation of the weak transport-entropy inequality \eqref{Tparen}:   for any real function 
$\varphi$ on $G$ and for any 
  $0< \alpha<1$, 
\begin{eqnarray}\label{syminf2}
\left(\int e^ {\alpha  \wideparen Q\varphi} d\mu\right)^{1/\alpha}\left(\int e^{-(1-\alpha)\varphi} d\mu\right)^{1/(1-\alpha)}\leq 1,
\end{eqnarray}
where  the infimum convolution operator $\wideparen Q\varphi$  is defined by
\[ \qquad \quad \wideparen Q\varphi(\sigma)=\inf_{p\in \P(G)} \left\{  \int \varphi\,dp  +\frac 1{2c(\ell)^2} \sum_{k=1}^n \left(\int \1_{\sigma(k)\neq y(k)} \,dp(y)\right)^2 \right\}, \quad \sigma\in G.\]
As explained at the end of this section, the property \eqref{syminf2} directly provides the following 
version of the Talagrand's concentration result for any measure on  $G$ of the set $\mathcal M_{\mathcal T}$. 
\begin{corollary}\label{talgen} Let $G$ be a  group of permutations  with $\ell$-local base $\mathcal T$. Let $\mu\in \P(G)$ be a  measure  of the set $\mathcal M_{\mathcal T}$ defined by \eqref{setmeas}. Assume that $\mu$ and $G$ satisfy the conditions of $(b)$ in Theorem  \ref{transportsymetrique}. Then, for all $A \subset G $ and all $\alpha\in(0,1)$, one has  
\[
\int e^{\frac{\alpha}{2c(\ell)^2} f(\sigma,A)}\,d\mu(\sigma) \leq\frac1{\mu(A)^{\alpha/(1-\alpha)}},
\]
with the same definition for $c(\ell)^2$ as in part $(b)$ of Theorem \ref{transportsymetrique}.
As a consequence, by Tchebychev inequality, for any $\alpha\in(0,1)$ and all $t\geq 0$, 
\[\mu\big(\{\sigma\in G,  f(\sigma,A)\geq t\}\big)\leq \frac {e^{-\frac{\alpha t}{2c(\ell)^2}}}{\mu(A)^{\alpha/(1-\alpha)}}.\]
\end{corollary}
For $\alpha=1/2$ and  $\mu=\mu_o$  the uniform law on a $\ell$-local group of $G$,  this result is exactly Theorem 2.1 by Luczak-McDiarmid \cite{LD03}, that generalizes Theorem \ref{tala95} on $S_n$ (since $S_n$  is a 2-local group). 
\end{itemize}

By projection arguments,  Theorem \ref{transportsymetrique} applied with the uniform law $\mu_o$ on the symmetric group $S_n$, also provides transport-entropy inequalities for the uniform law on the slices of the discrete cube $\{0,1\}^n$. Namely, for $n\geq 1$, let us denote by $\X_{k,n-k}$,  $k\in\{0,\dots,n\}$, \textit{the slices of discrete cube}  defined by 
\[\X_{k,n-k}:=\left\{ x=(x_1,\ldots,x_n)\in\{0,1\}^n, \sum_{i=1}^n x_i=k\right\}.\] 
The uniform law on $\X_{k,n-k}$, denoted by $\mu_{k,n-k}$,  is the pushed forward of $\mu_o$ by the projection map 
\[\begin{array}{cccl}&S_n&\to& \X_{k,n-k}\\
P:&\sigma &\mapsto &\1_{\sigma([k])},\end{array}\]
where $\sigma([k]):=\{\sigma(1),\ldots,\sigma(k)\}$ and for any subset $A$ of $[n]$, $\1_A$ is the vector with coordinates $\1_A(i),i\in [n]$. 
In other terms, $\mu_{k,n-k}=P\#\mu_o$ and $\mu_{k,n-k}(x)={\binom{n}{k}}^{-1}$
for all $x\in \X_{k,n-k}$.
Let $d_h$ denotes the \textit{Hamming distance on $\X_{k,n-k}$} defined by
\[d_h(x,y):=\frac12\sum_{i=1}^n \1_{x_i\neq y_i},\qquad x,y\in \X_{k,n-k}.\]
\begin{theorem}\label{transportslice} Let  $\mu_{k,n-k}$ be the uniform law on $\X_{k,n-k}$, a slice of the discrete cube.
\begin{enumerate} 
\item[(a)] For all probability measures  $\nu_1$ and $\nu_2$ on $\X_{k,n-k}$,
\begin{eqnarray*}
\frac2{C_{k,n-k}} {W}_1^2(\nu_1,\nu_2)\leq \left(\sqrt{H(\nu_1|\mu_{k,n-k})}+\sqrt{H(\nu_2|\mu_{k,n-k})}\right)^2,
\end{eqnarray*}
and
\begin{eqnarray*}
\frac1{2C_{k,n-k}} {\widetilde T}_2(\nu_2|\nu_1)\leq \left(\sqrt{H(\nu_1|\mu_{k,n-k})}+\sqrt{H(\nu_2|\mu_{k,n-k})}\right)^2,
\end{eqnarray*}
where $W_1$ is the Wasserstein distance associated to $d_h$, ${\widetilde T}_2$ is the weak optimal transport cost defined by \eqref{defTtilde} with $d=d_h$, and 
$\C_{k,n-k}=\min(k,n-k)$.

\item[(b)] 
For all probability measures  $\nu_1$ and $\nu_2$ on $\X_{k,n-k}$,
 \begin{equation}\label{That}
\frac1{8}  \widehat \T_2(\nu_2|\nu_1)\leq \left( \sqrt{H(\nu_1|\mu_{k,n-k})} +\sqrt{  H(\nu_2|\mu_{k,n-k})}\right)^2,
\end{equation}
where  
\begin{eqnarray*}
  {\widehat T}_2(\nu_2|\nu_1):=\inf_{\pi\in \Pi(\nu_1,\nu_2)}\int \sum_{i=1}^n \left(\int \1_{x_i\neq y_i} dp_x(y)\right)^2 d\nu_1(x), 
  \end{eqnarray*}
with $\pi(x,y)= \nu_1(x) p_x(y)$ for all $x,y\in \X_{k,n-k}$.
\end{enumerate}
\end{theorem}

Up to constants, the weak transport inequality \eqref{That} is the stronger one since
for all $\nu_1,\nu_2\in \P(\X_{k,n-k})$,
\[W_1^2(\nu_1,\nu_2)\leq  {\widetilde T}_2(\nu_2|\nu_1)\leq \frac n4 \, {\widehat T}_2(\nu_2|\nu_1).\]
The proof of Theorem \ref{transportslice} is given in section \ref{sectionslice}. The transport-entropy inequality \eqref{That} is derived by projection from  the transport-entropy inequality  \eqref{Tparen} for the uniform measure $\mu_o$ on $S_n$. The same projection argument  could be used to reach the results of (a) from the transport-entropy inequality of (a) in Theorem \ref{transportsymetrique}, but it provides  worse constants. The constant $C_{k,n-k}$  is obtained by working directly on $\X_{k,n-k}$ and following  similar arguments as in the proof of Theorem \ref{transportsymetrique}.

\vspace{0,3 cm}
 \textbf{ Remark :}
The results of Theorem \ref{transportslice} also extend to the multinomial law. Let $E=\{e_1,\ldots,e_m\}$ be a set of cardinality $m$ and let $k_1,\ldots, k_m$ be a collection of non-zero integers satisfying $k_1+\cdots+k_m=n$. \textit{The multinomial law} $\mu_{k_1,\ldots, k_m}$ is by definition the uniform law on the set 
\[\X_{k_1,\ldots, k_m}:=\bigg\{x\in E^n, \mbox{ such that  for all } l\in[m], \#\big\{i\in[n], x_i=e_l\big\}=k_l \bigg\}.\]
For any $x\in \X_{k_1,\ldots, k_m}$, one has $\mu_{k_1,\ldots, k_m}(x)=\frac{k_1!\cdots k_m!}{n!}$.
 As a result, the weak transport-entropy  inequality \eqref{That} holds on $\X_{k_1,\ldots, k_m}$ replacing the measure $\mu_{k,n-k}$ by the measure $\mu_{k_1,\ldots, k_m}$. The proof of this result  is a simple generalization of the one on $\X_{k,n-k}$, by using the projection map $P:S_n\to \X_{k_1,\ldots, k_m}$ defined by: $P(\sigma)=x$ if and only if
 \[x_i=e_l,\quad \forall l\in [m], \,\forall i\in J_l,\]
 where $J_l:=\big  \{i\in [n], k_0+\cdots+k_{l-1}<  i\leq k_0+\cdots+k_{l}\big\}$, with $k_0=0$. The details of this proof are  left to the reader.

A straightforward application of transport-entropy inequalities  is deviation's bounds for different classes of functions. For more comprehension, we  present below  deviations bounds that can be reached from Theorem \ref{transportsymetrique} for any measure in ${\mathcal M}_{\mathcal T}$. A similar  corollary can be derived from Theorem \ref{transportslice} on the slices of the discrete cube. 

For any $h:G\to \R$, the mean of $h$ is denoted by $\mu(h):=\int h \,d\mu$.
 
\begin{corollary}{}\label{cordev} Let $G$ be a  group of permutations  with  $\ell$-local base $\mathcal T$, $G\neq \{id\}$. Let $\mu\in \P(G)$ be a  measure  of the set $\mathcal M_{\mathcal T}$  defined by \eqref{setmeas}. Let $g$ be a real function on $G$.
\begin{enumerate}
\item[(a)]
Assume that there exists a function $\beta:G\to\R^+$  such that for all   $\tau,\sigma\in G$, 
\[g(\tau)-g(\sigma)\leq  \beta(\tau) d(\tau,\sigma),\]
where $d=d_T$ or $d=d_H$.
Then  for all  $u\geq 0$, one has 
\[\mu\left(g\geq  \mu(g)+u\right)\leq \exp\left( -\,\frac{2 u^2}{ K_n c(\ell)^2\,\sup_{\sigma\in G}  \beta(\sigma)^2 }\right).\]
and
\[\mu\left(g\leq \mu( g)-u\right)\leq \exp\left( -\,\frac{2 u^2}{ K_n c(\ell)^2\,\min(\sup_{\sigma\in G}\beta(\sigma)^2, 4\mu(\beta^2))  }\right),\]
where the constants $c(\ell)$ and $K_n$ are defined as in  part $(a)$ of Theorem \ref{transportsymetrique}.

\item[(b)] Assume that $\mu$ and $G$ satisfy the conditions of $(b)$ in Theorem  \ref{transportsymetrique}. Let $g$ be a so-called  \textit{ configuration function}. This  means  that
 there exist  functions $\alpha_k:G\to\R^+$, $k\in\{1,\ldots,n\}$ such that for all   $\tau,\sigma\in G$, 
\[g(\tau)-g(\sigma)\leq \sum_{k=1}^n \alpha_k(\tau) \1_{\tau(k)\neq \sigma(k)}.\]

Then, for all  $v\geq 0, \lambda\geq 0$, one has 
\[\mu\left(g\geq  \mu(g)+v +\frac{\lambda c(\ell)^2 |\alpha|_2}{2}\right)\leq e^{-\lambda v},\]
and for all  $u\geq 0$,
\[\mu\left(g\leq \mu( g)-u\right)\leq \exp\left( -\,\frac{u^2}{2 c(\ell)^2\mu\left(  |\alpha|_2^2\right) }\right),\]
where $\displaystyle|\alpha(\sigma)|_2^2:=\sum_{k=1}^n \alpha^2_k(\sigma)$ and  $c(\ell)$ is defined as in  part $(b)$ of Theorem~\ref{transportsymetrique}.
We also have,  for all  $u\geq 0$
\[\mu\left(g\geq  \mu(g)+u\right)\leq \exp\left( -\,\frac{u^2}{ 2 c(\ell)^2\,\sup_{\sigma\in G}  |\alpha(\sigma)|_2^2 }\right),\]
and if there exists $M\geq 0$ such that $|\alpha|_2^{2}\leq M g$, then for all  $u\geq 0$
\[\mu\left(g\geq  \mu(g)+u\right)\leq \exp\left( -\,\frac{u^2}{ 2 c(\ell)^2 M(\mu(g)+u) }\right),\]
\end{enumerate}
\end{corollary}

 \textbf{ Comments and examples:}
\begin{itemize}
\item The above deviation's bounds of $g$  around its  mean $\mu(g)$ are directly derived from the dual representations \eqref{Hoeff},\eqref{symtilde},\eqref{syminf2} of the transport-entropy inequalities of Theorem \ref{transportsymetrique}, when $\alpha$ goes to 0 or $\alpha$ goes to 1.
By classical arguments (see \cite{Led01}),  Corollary \ref{cordev} also implies   deviation's bounds around a median $M(g)$ of $
g$, but  we loose in the constants with this procedure. However, starting directly from Corollary \ref{talgen},  we get the following  bound under the assumption of $(b)$: for all  $u\geq 0$,
\begin{eqnarray}\label{devmed}
\mu(g\geq M(g)+u)\leq \frac12 \exp\left(-w\left(\frac u{\sqrt 2 c(\ell) \sup_{\sigma\in G}|\alpha(\sigma)|_2}\right)\right),
\end{eqnarray}
where $w(u)= u( u-2\sqrt{\log 2})$,    $u\geq 0$.

The idea of the proof is to choose the set $A=\{\sigma\in G, g(\sigma)\leq M(g)\}$ of measure $\mu(A)\geq 1/2$ and to show that the asumption of $(b)$ implies 
\[ \qquad\big\{\sigma\in G, f(\sigma,A)<t\big\}\subset  \left\{\sigma\in G, g(\sigma)<M(g)+t\sup_{\sigma\in G}|\alpha(\sigma)|_2\right\},\quad t\geq 0.\]
Then, the deviation bound above the median directly follows from Corollary \ref{talgen} by optimizing over all $\alpha\in(0,1)$.
With identical arguments,  the same bound can be reached for    $\mu(g\leq  M(g)-u)$.

\item In (a), the bound above the mean is a  simple consequence of \eqref{Hoeff}. As settled in (a), this bound also holds for the deviations under the mean, and it  can be   slightly improved  by replacing   $ \sup_{\sigma\in G}  \beta(\sigma)^2$ by 
$4 \mu(\beta^2)$. This small improvement is a consequence of the weak transport inequality with stronger cost $\widetilde T_2$. The same kind of improvement could be  reached for the deviations above the mean under additional Lipschitz regularity conditions on the function $\beta$.  
\item Let $\varphi:[0,1]^n\to\R$ be a 1-Lipschitz convex function and let $x=(x_1,\ldots,x_n)$ be a fixed vector of $[0,1]^n$. For any $\sigma\in G$, let $x_\sigma:=(x_{\sigma(1)},\ldots, x_{\sigma(n)})$.
By applying the results of  $(b)$ (or even \eqref{devmed}) to the particular function $g_x(\sigma)=\varphi(x_\sigma)$, $\sigma\in G$,  we recover and extend to any group $G$ with $\ell$-local base $\mathcal T$  and to any measure in $\mathcal M_{\mathcal T}$ satisfying \eqref{hypoinv},  the deviation inequality by  Adamczak, Chafa\"{\i} and Wolff  \cite{ACW14} (Theorem 3.1) obtained from Theorem \ref{tala95} by Talagrand. Namely, since for any $\sigma,\tau\in G$, 
\[\varphi(x_\tau)-\varphi(x_\sigma)\leq \sum_{k=1}^n \partial_k \varphi(x_{\tau})(x_{\tau(k)}-x_{\sigma(k)})\leq \sum _{k=1}^n |\partial_k \varphi(x_\tau) |\1_{\tau(k)\neq \sigma(k)}, \]
with $\sum_{k=1}^n |\partial_k \varphi(x_\tau) |^2=|\nabla \varphi(x_\tau) |^2\leq 1$, 
 Corollary  \ref{cordev} implies, for any choice of vector $x=(x_1,\ldots,x_n)\in [0,1]^n$,
 \[\mu(|g_x- \mu(g_x)| \geq u )\leq 2\exp\left( -\,\frac{u^2}{ 2 c(\ell)^2  }\right), \quad u\geq 0.\]
This concentration property on $S_n$ (with $\ell=2$) plays a key role  in the approach by Adamczak and al. \cite{ACW14}, to study  the convergence of the  empirical spectral measure of random matrices with exchangeable entries, when the size of the matrices is increasing. 
\item As a second example, for any $t$ in a finite set $\mathcal F$, let  $(a_{i,j}^t)_{1\leq i,j\leq n}$ be a collection of non negative real numbers and consider the function 
\[g(\sigma)= \sup_{t\in \mathcal F} \left(\sum_{k=1}^n a^t_{k,\sigma(k)}\right),\qquad \sigma\in G.\]
This function satisfies, for any $\sigma, \tau \in G$, 
\[g(\tau)-g(\sigma)\leq \sum_{k=1}^n \left(a^{t(\tau)}_{k,\tau(k)}-a^{t(\tau)}_{k,\sigma(k)}\right) \1_{\tau(k)\neq \sigma(k)}\leq \sum_{k=1}^na^{t(\tau)}_{k,\tau(k)}\1_{\tau(k)\neq \sigma(k)},\]
where $t(\tau)\in \mathcal F$ is chosen so that 
\[g(\tau)= \sum_{k=1}^n a^{t(\tau)}_{k,\tau(k)}.\]
Let us consider the function  
\[h(\sigma)= \sup_{t\in \mathcal F}  \left(\sum_{k=1}^n (a^t_{k,\sigma(k)})^2\right),\qquad \sigma\in G.\]
The mean of $h$, $\mu(h)$, can be interpreted as a variance term as regards to $g$.
Observing that $g$ satisfies the condition of (b) with 
\[\alpha_k(\tau):=a^{t(\tau)}_{k,\tau(k)},\]
and $|\alpha|_2^2\leq h$, 
 Corollary  \ref{cordev} provides the following Bernstein deviation's bounds, for all $u\geq 0$, 
 \[\mu\left(g\leq \mu( g)-u\right)\leq \exp\left( -\,\frac{u^2}{2 c(\ell)^2\mu\left(  h\right) }\right),\]
 and for all $\lambda, v\geq 0$,
 \[\mu\left(g\geq  \mu(g)+v +\frac{\lambda c(\ell)^2 h}{2}\right)\leq e^{-\lambda v}.\]
 If the real numbers $a_{i,j}$ are bounded by $M$, then $|\alpha|_2^2\leq Mg$ and therefore Corollary  \ref{cordev} also provides for all $u\geq 0$,
 \[\mu\left(g\geq  \mu(g)+u\right)\leq \exp\left( -\,\frac{u^2}{ 2 c(\ell)^2 M(\mu(g)+u) }\right).\]
 If we want to bound the deviation above the mean in terms of the variance term $\mu(h)$, it suffises to observe that the last inequality provides deviations bounds for the function $h$, replacing $g$ by $h$ and $M$ by $M^2$. Then,  as a consequence of all the above deviation's results, it follows that for all $\lambda, v,\gamma\geq 0$,
 \begin{align*}
 &\mu\left(g\geq \mu(g) +v+\frac{\lambda c(\ell)^2(\mu(h)+\gamma)}{2}\right)\\
 &\leq \mu\left(g\geq \mu(g) +v+\frac{\lambda c(\ell)^2h}2\right)+\mu(h\geq \mu(h)+\gamma)\\
 &\leq e^{-\lambda v}+\exp\left( -\,\frac{\gamma^2}{ 2 c(\ell)^2 M^2(\mu(h)+\gamma) }
 \right).
 \end{align*}
 By choosing $\gamma=Mu$, $\lambda=\frac u{c(\ell)^2 M^2(\mu(h)+Mu)}$, and $v=u/2$, we get the following Bernstein deviation inequality for the deviation of $g $ above its mean, for all $u\geq 0$
 \[\mu(g\geq \mu(g)+u)\leq 2 \exp\left( -\,\frac{u^2}{ 2 c(\ell)^2 (\mu(h)+Mu) }\right).\]
 All the previous deviation's inequalities extend to countable sets $\mathcal F$ by monotone convergence.
 
When $\mathcal F$ is reduced to a singleton, these deviation's results simply implies Bernstein deviation's results for $g(\sigma)=\sum_{k=1}^n a_{k,\sigma(k)}$ when $-M\leq a_{i,j}\geq M$ for all $1\leq i,j\leq n$, by following for example the procedure presented in \cite[Section 4.2]{BDR15}. Thus, we extend the deviation's results of \cite{BDR15} to probability measures in ${\mathcal M}_{\mathcal T}$.
\item As a last example, let $g(\sigma)=|\sigma|_l$ denotes the number of cycles of lenght $l$ in the cycle decomposition of a permutation $\sigma$. Let us show that $g$ is a configuration function. Let ${\mathcal C}_l(\tau)$ denotes the set of cycles of  lenght $l$ in the cycle decomposition of a permutation $\tau$. One has 
\begin{align*}
|\tau|_l&=\#\{{\mathcal C}_l(\tau)\cap {\mathcal C}_l(\sigma)\}+\#\{c\in {\mathcal C}_l(\tau), \mbox{ such that }c\notin {\mathcal C}_l(\sigma)\}\\
&\leq |\sigma|_l + \#\{c\in {\mathcal C}_l(\tau), \mbox{ such that }c\notin {\mathcal C}_l(\sigma)\}.
\end{align*}
 If $c\in {\mathcal C}_l(\tau)$ and $c\notin {\mathcal C}_l(\sigma)$ then there exists $k$ in the support of $c$  such that $\tau(k)\neq \sigma(k)$. 
 As a consequence, one has 
\[\#\{c\in {\mathcal C}_l(\tau), \mbox{ such that }c\notin {\mathcal C}_l(\sigma)\}
\leq \sum_{k=1}^n \alpha_k(\tau)\1_{\sigma(k)\neq \tau(k)}, \]
where $\alpha_k(\tau)=1$ if $k$ is in the support of a cycle of lenght $l$ of the cycle decomposition of $\tau$, and  $\alpha_l(\tau)=0$ otherwise.
Thus, we get that the function $g$ satisfies the condition of $(b)$, $g$ is a configuration function. Finally, observing that 
$|\alpha|_2^2=lg$, Corollary  \ref{cordev}
provides for any measure $\mu\in {\mathcal M}_{\mathcal T}$ satisfying \eqref{hypoinv}, for all $u\geq 0$, 
\[\mu\left(g\leq \mu( g)-u\right)\leq \exp\left( -\,\frac{u^2}{2 c(\ell)^2l\mu\left(  h\right) }\right),\]
and
\[\mu\left(g\geq  \mu(g)+u\right)\leq \exp\left( -\,\frac{u^2}{ 2 c(\ell)^2 l(\mu(g)+u) }\right).\]
\item The aim of this paper is to clarify the links between Talagrand's type of  concentration results on the symmetric group and functional inequalities derived from the transport-entropy inequalities. For brevity's sake, applications of these functional inequalities  are not fully developped in the present  paper. However, let us briefly mention some other applications using concentration results on the symmetric group: the stochastic travelling salesman problem for  sampling without replacement (see Appendix \cite{Pau14}), graph coloring problems (see \cite{McD02}). We also refer to the surveys and books  \cite{DP09,MR02} for other numerous examples of application of the concentration of measure principle  in randomized algorithms.

\end{itemize}
\begin{proof}[Proof of Corollary \ref{cordev}] We start with the proof of (b). From the assumption on  the function $g$, we get that for any  $p\in \P(G)$ 
\begin{multline*}
\int g\, dp\geq g(\sigma)-\sum_{k=1}^n \left(\alpha_k(\sigma) \int \1_{\sigma(k)\neq \tau(k)} \,dp(\tau)\right)\\
\geq g(\sigma)-|\alpha(\sigma)|_2 \left(\sum_{k=1}^n \left(\int \1_{\sigma(k)\neq \tau(k)} dp(\tau)\right)^2\right)^{1/2}.
\end{multline*}
Let $\lambda\geq 0$. Plugging this estimate into the definition of $\wideparen Q(\lambda g)$, it follows that  for any $\sigma\in G$
\[\wideparen Q(\lambda g)(\sigma)\geq \lambda g(\sigma)-\sup_{u\geq 0}\left\{ \lambda|\alpha(\sigma)|_2u -\frac{u^2}{2c(\ell)^2}\right\}=\lambda g(\sigma)-\frac{\lambda^2|\alpha(\sigma)|_2^2 c(\ell)^2}2.\]
As $\alpha$ goes to 1, \eqref{syminf2} applied to the function $\lambda g$ yields
\[\int e^{\wideparen Q(\lambda g)} d\mu\leq e^{\lambda \mu(g)}, \]
and therefore 
\begin{eqnarray}\label{exposup0}
\int \exp\left(\lambda g-\frac{\lambda^2c(\ell)^2 |\alpha|_2^2 }{2}\right) d\mu\leq e^{\lambda \mu(g)},
\end{eqnarray}
\begin{eqnarray}\label{exposup}
\int e^{\lambda g} d\mu\leq \exp\left(\lambda \mu(g)+ \frac{\lambda^2c(\ell)^2 \sup_{\sigma\in G}|\alpha(\sigma)|_2^2 }{2}\right),
\end{eqnarray}
and if $|\alpha|_2^2\leq Mg$, 
\begin{eqnarray}\label{exposup1}
\int \exp\left(\lambda \left(1-\frac{\lambda c(\ell)^2 M }{2}\right)g \right) d\mu\leq e^{\lambda \mu(g)}.
\end{eqnarray}
 As $\alpha$ goes to 0, \eqref{syminf2}  yields
\[\int e^{-\lambda g} d\mu\leq e^{\lambda \mu(\wideparen Q(\lambda g))}, \]
and therefore 
\begin{eqnarray}\label{expoinf}
\int e^{-\lambda g} d\mu \leq \exp\left({-\lambda \mu(g)+ \frac{\lambda^2c(\ell)^2\mu(|\alpha|_2^2) }2}\right).
\end{eqnarray}
The deviation bounds of (b) follows from \eqref{exposup0}, \eqref{expoinf}, \eqref{exposup}, \eqref{exposup1}  by Tchebychev inequality, and by optimizing over all $\lambda\geq 0$.

The deviation bounds of (a) are similarly obtained from \eqref{symtilde} by Tchebychev inequality. As above, the improvement for the deviation under the mean is a consequence of  \eqref{symtilde} applied to $\lambda g$, as $\alpha$ goes to 0, and  using the estimate
\[\widetilde Q_{K_n}(\lambda g)(\sigma)\geq \lambda g(\sigma)-\frac{\lambda^2 \beta(\sigma)^2 c(\ell)^2 K_n}2.\]

\end{proof}

\begin{proof}[Proof of Corollary \ref{talgen}]
Take a subset $A \subset G$ and consider the function $\varphi_\lambda$ which takes the values $0$ on $A$ and $\lambda>0$ on $G \setminus A$. 
It holds
\begin{align*}
 \wideparen Q\varphi_\lambda(\sigma)& = \inf_{p \in \mathcal{P}(G)}\left\{ \lambda(1-p(A)) + \frac 1{2c(\ell)^2}\sum_{j=1}^n \left(  \int \1_{\sigma(j)\neq y(j)}\, dp(y)\right)^2 \right\}\\
& = \inf_{\beta \in [0,1]} \{ \lambda (1-\beta) + \psi(\beta,\sigma)\},
\end{align*}
denoting by
\[
\psi(\beta,\sigma) = \inf \left\{ \frac 1{2c(\ell)^2}\sum_{j=1}^n \left(  \int \1_{\sigma(j)\neq y(j)}\, dp(y)\right)^2 ; p(A)=\beta \right\}.
\]
So it holds
\begin{align*}
 \wideparen Q\varphi_\lambda(\sigma) &= \min\left( \inf_{\beta \in [0,1-\varepsilon]} \{ \lambda (1-\beta) + \psi(\beta,\sigma)\} ;   \inf_{\beta \in [1-\varepsilon,1]} \{ \lambda (1-\beta) + \psi(\beta,\sigma)\},\right) \\
&\geq  \min\left( \lambda \varepsilon ; \inf_{\beta \geq 1-\varepsilon} \psi(\beta,\sigma)       \right) \to  \inf_{\beta \geq 1-\varepsilon} \psi(\beta,\sigma),
\end{align*}
as $\lambda \to \infty.$
It is easy to check that for any fixed $\sigma$, the function $\psi(\,\cdot\,,\sigma)$ is continuous on $[0,1]$, so letting $\varepsilon$ go to $0$, we get 
$\liminf_{\lambda \to \infty}  \wideparen Q \varphi_\lambda(\sigma) \geq \psi(1,\sigma).$
On the other hand, $  \wideparen Q \varphi_\lambda(\sigma) \leq \psi(1,\sigma)$ for all $\lambda>0$. This proves that $\lim_{\lambda \to \infty}  \wideparen Q \varphi_\lambda(\sigma) = \psi(1,\sigma)$.
Applying \eqref{syminf2} to $\varphi_\lambda$ and letting $\lambda$ go to infinity yields to
\[
\int e^{\alpha \psi(1,\sigma)}\,d\mu\cdot \mu(A)^{\alpha/(1-\alpha)} \leq 1.
\]
It remains to observe that  $\psi(1,\sigma)=\frac {f(\sigma,A)}{2c(\ell)^2}$. 
\end{proof}

\section{Proof of Theorem \ref{transportsymetrique}}
Let  $\mathcal T_n=(t_{i_j,j},j\in\{2,\ldots,n\}, i_j\in O_j)$ be a $\ell$-local base of $G$. Let $\mu$ be a probability measure of the set $\mathcal M_{{\mathcal T}_n}$  given by \eqref{setmeas}. Then, there exists a product probability measure $\hat \nu=\hat \nu_1\otimes \cdots \otimes \hat\nu_n$ such that $\mu=U_{\mathcal T_n}\# \hat \nu$ where the map $U_{\mathcal T_n}$ is given by \eqref{TT}.

Each transport-entropy inequality of Theorem \ref{transportsymetrique} is obtained by induction over $n$ and using the  partition $(H_i)_{i\in\orb(n)}$ of the group $G$ defined by: for any $i\in orb(n)=O_n$, 
\begin{eqnarray}\label{partition}
H_i:=\left\{\sigma\in G, \sigma(i)=n \right\}.
\end{eqnarray}
According to our notations, $H_n=G_{n-1}$ is a subgroup of $G$, and we may easily check that  $\mathcal T_{n-1}$ is a $\ell$-local base of this subgroup. We also observe that  if $G$ is a normal subgroup of $S_n$ then $G_{n-1}$ is a normal subgroup of $S_{n-1}$.

Moreover,  for any $i\in O_n$, $H_i$ is the coset defined by $H_i=H_nt_{in}$.
From the definition of $\mu$, if $\sigma\in H_i$, then there exist $i_2,\ldots, i_{n-1}$ such that $\sigma=t_{i_2,2}\cdots t_{i_{n-1}, n-1} t_{i,n}$ and therefore 
\[\mu(\sigma)=\hat \nu_2(i_2)\cdots\hat\nu_{n-1}(i_{n-1}) \hat\nu_n(i).\]
As a consequence, one has $\mu(H_i)=\hat \nu_n(i)$. 
Let $\mu_i$ denote the restriction of $\mu$ to $H_i$ defined by
\[\mu_i(\sigma)=\frac{\mu(\sigma)}{\mu(H_i)}\, \1_{\sigma\in H_i}.\]
From the construction of $\mu$, $\mu_n=U_{\mathcal T_{n-1}}\# (\hat\nu_1\otimes\cdots\otimes\hat  \nu_{n-1})$. Moreover, for all $\sigma \in H_n$,  one has $\sigma t_{i,n}\in H_i$ and 
\begin{eqnarray}\label{inducmes}
\mu_n(\sigma)=\frac{\mu(\sigma)}{\mu(H_n)}=\frac{\mu(\sigma t_{i,n})}{\mu(H_i)}=\mu_i(\sigma t_{i,n}).
\end{eqnarray}
Moreover if $\mu$ satisfies the condition \eqref{hypoinv}, then $\mu_n\in \P(G_{n-1})$ satisfies the same condition at rank $n-1$: namely, for any $\sigma\in G_{n-1}$,  $t\in S_{n-1}$,
\[\mu_n(\sigma)=\mu_n(\sigma^{-1})\quad\mbox{ and }\quad \mu_n(\sigma)=\mu_n(t^{-1}\sigma t).\]  These properties are  needed in the induction step of the proofs.

When $G$ is a $\ell$-local group, let us note that if $i$ and $l$ are elements of $O_n=\orb(n)$, then from the $\ell$-local property, there exists $t_{i,l}\in G$ such that $t_{i,l}(i)=l$ and $\degr(t_{i,l})\leq \ell$. We also have  $H_l=H_it_{i,l}$. If moreover  $\mu=\mu_o$ is  the uniform law on $G$, then  for any $i,l\in O_n$, $\mu_i(H_i)=\mu_l(H_l)=\frac{1}{\# O_n}$. In that case we will  use in the proofs  the following property: for any $\sigma\in H_n$, one has 
$\sigma t_{i,n}\in H_i$, $\sigma t_{i,n}t_{i,l}^{-1}\in H_l$, and
\begin{eqnarray}\label{inducmesbis}
\mu_n(\sigma)=\frac{\#O_n}{\# G}=\mu_i(\sigma t_{i,n})=\mu_l(\sigma t_{i,n}t_{i,l}^{-1}).
\end{eqnarray}
The measure $\mu_n$ is the uniform measure on the $\ell$-local subgroup $H_n=G_{n-1}$.
\begin{proof}[Proof of  (a) in  Theorem \ref{transportsymetrique}] 
 As already mentioned, since  $W_1$ satisfies a triangular inequality, the transport-entropy inequality \eqref{TW1} is equivalent to the following one: for all  $\nu\in{\mathcal P}(G)$, 
\[ \frac{2}{c(\ell)^2} W^2_1(\nu,\mu)\leq K_{n}\,H(\nu|\mu).\]
A dual formulation of this property given by Theorem 2.7 in \cite{GRST16} and Proposition 3.1 in \cite{Sam16}  is the following: for all  functions  $\varphi$ on  $G$ and all  $\lambda\geq 0$,
\begin{eqnarray}\label{transn}
\int e^{\lambda  Q\varphi} d\mu\leq e^{\int \lambda \varphi \,d\mu +K_{n}c(\ell)^2\lambda^2/8},
\end{eqnarray}
with  
\[Q\varphi(\sigma)=\inf_{p\in \P(S_n)}\left\{\int \varphi dp +\int d(\sigma,\tau) \,dp(\tau)\right\}\]
We will prove the inequality  \eqref{transn}  by induction on $n$. 

Assume that $n=2$. If $G=\{id\}$ then $K_n=0$ and  the inequality \eqref{transn}  is obvious.  If $G\neq\{id\}$, then $G$  is the two points space, $G=S_2$, $\ell=2$ and  one has    
\[Q\varphi(\sigma)=\inf_{p\in \P(S_2)}\left\{\int \varphi dp +c(2) \int \1_{\sigma\neq \tau} \,dp(\tau)\right\}.\]
In that case,   \eqref{transn} exactly corresponds  to the following dual form of the Csiszar-Kullback-Pinsker inequality \eqref{CKP} (see  Proposition 3.1 in \cite{Sam16} ): for any probability measure  $\nu$ on a Polish space  $\mathcal X$, for any measurable function  $f:\mathcal{X}\to \R $,
\begin{eqnarray}\label{Pinskerdual}
\int e^{\lambda R^cf} d\nu\leq e^{\lambda \int f \,d\nu + \lambda^2c^2/8},\qquad\forall \lambda,c\geq 0,
\end{eqnarray} 
with 
$\displaystyle
R^c f(x)=\inf_{p\in \P(\mathcal{X})} \left\{\int f dp+c\int \1_{x\neq y} dp(y)\right\}, x\in \X.
$

The induction step will be  also a consequence of  \eqref{Pinskerdual}.
Let $(H_i)_{i\in O_n}$ be the partition of  $G$ defined by \eqref{partition}.
Any  $p\in \P(G)$ admits a unique decomposition  defined by 
\begin{eqnarray}\label{dec}
\qquad p=\sum_{i\in  O_n} \hat p(i) p_i,  \qquad \mbox{ with  }\quad p_i\in \P(H_i) \qquad \mbox{ and } \quad  \hat p(i)=p(H_i).
\end{eqnarray}
This decomposition defines a  probability measure $\hat p$ on $O_n$.
In particular, according to the definition of the measure  $\mu\in \mathcal M_{{\mathcal T}_n}$ and  since   $\hat\nu_n(i)=\mu(H_i)$, one has
\[\mu= \sum_{i\in O_n}  \hat \nu_n(i) \,\mu_i.\]
It follows that 
\[\int e^{\lambda Q\varphi} d\mu = \sum_{i\in O_n} \hat \nu_n(i)  \int e^{\lambda Q\varphi(\sigma)} d\mu_i(\sigma)=\sum_{i\in O_n}\hat \nu_n(i) 
\int e^{\lambda Q\varphi(\sigma t_{i,n})} d\mu_n(\sigma),\]
where the last equality is a consequence of property \eqref{inducmes}.
Now, we will bound the right-hand side of this equality by using the induction hypotheses.

For any function $g:G\to \R $ and any $t\in G$, let $g^t:G\to\R$ denote the function defined by $g^t(\sigma):=g(\sigma t)$.

For any function  $f:H_n\to \R$ and any $\sigma\in H_n$, let us note 
\[Q^{H_n} f(\sigma):= \inf_{p\in \P(H_n)}\left\{\int f \,dp +\int d(\sigma,\tau) \,dp(\tau)\right\}.\]
The next step of the proof relies on the following Lemma.
\begin{lemma}\label{gateau} Let $i\in O_n$, for any function $\varphi: H_i\to \R$ and any $\sigma\in H_n$, one has
\begin{enumerate}
\item \label{lemme21(1)} $\displaystyle Q\varphi (\sigma t_{i,n})\leq \inf_{\hat p\in \P(O_n)}\left\{\sum_{l\in O_n} Q^{H_n} \varphi^{t_{n,l}}(\sigma) \hat p(l) +c(\ell)\sum_{l\in O_n} \1_{l\neq i} \hat p(l)\right\},$\\
where $c(\ell)=\min(2\ell-1,n) $ if $d=d_H$ and $c(\ell)=2$ if $d=d_T$ .
\item $\displaystyle Q\varphi (\sigma t_{i,n})\leq \inf_{\hat p\in \P(O_n)}\left\{\sum_{l\in O_n} Q^{H_n} \varphi^{t_{i,n}t_{i,l}^{-1}}(\sigma) \hat p(l) +c(\ell)\sum_{l\in O_n} \1_{l\neq i} \hat p(l)\right\},$\\
where $c(\ell)=\ell $ if $d=d_H$ and $c(\ell)=1$ if $d=d_T$, and $t_{i,l}$  denotes  an element of $G$ with $\degr(t_{i,l})\leq \ell$ and such that  $t_{i,l}(i)=l$.
\end{enumerate}
\end{lemma}
This lemma is obtained using the  decomposition \eqref{dec} of the  measures $p\in \P(G)$ on the $H_j$'s.  Let $\sigma\in H_n$.  By the triangular inequality and using the invariance by translation of the distance $d$, one has 
\begin{align*}
\int d(\sigma t_{i,n},\tau) \,dp(\tau)&=\sum_{l\in O_n} \int_{H_l} d(\sigma t_{i,n},\tau) dp_l(\tau)\hat p(l)\\
&\leq \sum_{l\in O_n} \ d(\sigma t_{i,n},\sigma t_{l,n})  \hat p(l) +\sum_{l\in O_n} \int_{H_l} d(\sigma t_{l,n},\tau) dp_l(\tau) \hat p(l)\\
&=\sum_{l\in O_n} \ d( t_{i,n}, t_{l,n})  \hat p(l) +\sum_{l\in O_n} \int_{H_l} d(\sigma ,\tau t_{l,n}^{-1}) dp_l(\tau) \hat p(l)\
\end{align*}
and therefore, since $d( t_{i,n}, t_{l,n})\leq c(\ell)$ with $c(\ell)=\min(2\ell-1,n)$ if $d=d_H$ and $c(\ell)=2$ if $d=d_T$,
\begin{eqnarray}\label{triangular}
\qquad\int d(\sigma t_{i,n},\tau) \,dp(\tau) \leq  \sum_{l\in O_n} \int_{H_l} d(\sigma,\tau t_{l,n}^{-1}) dp_l(\tau) \hat p(l) +c(\ell) \sum_{l\in O_n} \1_{l\neq i}   \hat p(l).
\end{eqnarray}
 It follows that 
\begin{align*}
Q\varphi(\sigma t_{i,n})&\leq \inf_{\hat p\in \P(O_n)}\inf_{p_l\in \P(H_l),l\in O_n}\\
&\qquad\qquad \left\{
\sum_{l\in O_n}\left[ \int \varphi \,dp_l+ \int_{H_l} d(\sigma,\tau  t_{l,n}^{-1}) dp_l(\tau)\right] \hat p(l)+c(\ell)\sum_{l\in O_n} \1_{l\neq i} \hat p(l)
\right\}\\
&= \inf_{\hat p\in \P({O_n})}\inf_{q_l\in \P(H_n),l\in O_n}\\
&\qquad\qquad \left\{
\sum_{l\in  O_n}\left[ \int \varphi^{t_{l,n}} \,dq_l+ \int_{H_n} d(\sigma ,\tau ) dq_l(\tau)\right] \hat p(l)+c(\ell)\sum_{l\in  O_n} \1_{l\neq i} \hat p(l)
\right\}\\
&=\inf_{\hat p\in \P( O_n)}\left\{\sum_{l\in  O_n} Q^{H_n} \varphi^{t_{l,n}}(\sigma) \hat p(l) +c(\ell)\sum_{l\in  O_n} \1_{l\neq i} \hat p(l)\right\}.
\end{align*} 
The proof of the second inequality of Lemma \ref{gateau} is similar, starting from the following triangular inequality
\begin{align}
\int d(\sigma t_{i,n},\tau) \,dp(\tau)&=\sum_{l\in  O_n} \int_{H_l} d(\sigma t_{i,n},\tau) dp_l(\tau)\hat p(l)\nonumber \\ 
&\leq \sum_{l\in  O_n} \int  d(\sigma t_{i,n},\tau t_{i,l}) dp_l(\tau)   \hat p(l) +\sum_{l\in  O_n} \int_{H_l} d(\tau t_{i,l},\tau) dp_l(\tau) \hat p(l)\nonumber \\
&=\sum_{l\in  O_n} \int  d(\sigma ,\tau t_{i,l}t_{i,n}^{-1}) dp_l(\tau)   \hat p(l) +\sum_{l\in  O_n}d( t_{i,l}, id) \hat p(l)\nonumber \\
&\leq  \sum_{l\in  O_n} \int_{H_l} d(\sigma,\tau t_{i,l}t_{i,n}^{-1}) dp_l(\tau) \hat p(l) +c(\ell) \sum_{l\in  O_n} \1_{l\neq i}   \hat p(l), \label{triangularbis}
\end{align}
with $c(\ell)=\ell$ if $d=d_H$ and $c(\ell)=1$ if $d=d_T$.
The end of the proof of the second inequality of Lemma \ref{gateau} is left to the reader.

The induction step of the proof of \eqref{transn} continues by applying consecutively  Lemma \ref{gateau} \eqref{lemme21(1)},   the H\"older inequality, and the induction hypotheses to the measure $\mu_n$ on the subgroup $H_n=G_{n-1}$ with $\ell$-local base ${\mathcal T}_{n-1}$.

If $O_n=\{n\}$ then $K_n=K_{n-1}$ and 
\begin{eqnarray*}
\int e^{\lambda Q\varphi} d\mu = 
\int e^{\lambda Q\varphi(\sigma)} d\mu_n(\sigma)\leq e^{\int\lambda \varphi d\mu_n +K_{n-1}c(\ell)^2/8}=e^{\int\lambda \varphi d\mu +K_nc(\ell)^2/8}\end{eqnarray*}

If $O_n\neq \{n\}$ then $K_n=K_{n-1}+1$ and for any $i\in  O_n$, 
\begin{align*}
&\int e^{\lambda Q\varphi(\sigma t_{i,n})} d\mu_n(\sigma)\leq \inf_{\hat p\in \P( O_n)}\left\{ \prod_{l\in O_n}\left(\int e^{\lambda Q^{H_n} \varphi^{t_{l,n}}} d\mu_n \right)^{\hat p(l)}e^{c(\ell) \lambda \sum_{l=1}^n \1_{l\neq i} \hat p(l)}\right\}\\
&\leq \exp\left[\inf_{\hat p\in \P( O_n)}\left\{ \lambda\sum_{l\in O_n} \left(\int \varphi^{t_{l,n}} d\mu_n\right)\hat p(l)+K_{n-1}c(\ell)^2\frac{\lambda^2}8 +c(\ell) \lambda \sum_{l\in O_n} \1_{l\neq i} \hat p(l)\right\}\right]\\
&= \exp \left[\lambda\inf_{\hat p\in \P( O_n)}\left\{\sum_{l\in O_n}\hat \varphi(l) \hat p(l)+c(\ell)  \sum_{l\in O_n} \1_{l\neq i} \hat p(l)\right\}+K_{n-1}c(\ell)^2\frac{\lambda^2}8 \right],
\end{align*}
where, by using property \eqref{inducmes}, $\hat \varphi(l):=\int \varphi d\mu_l=\int \varphi^{t_{l,n}} d\mu_n$. 
Let us consider again the above infimum-convolution $R^c\hat\varphi$ defined on the space $\mathcal{X}=  O_n$, with $c=c(\ell)$, one has  
\[R^c\hat\varphi(i)=\inf_{\hat p\in \P( O_n)}\left\{\sum_{l\in O_n} \hat \varphi(l) \hat p(l)+c  \sum_{l\in O_n} \1_{l\neq i} \hat p(l)\right\}.\]
By applying  \eqref{Pinskerdual} with the probability measure  $\nu=\hat \nu_n$ on  $ O_n$, the previous inequality gives 
\begin{align*}
&\int e^{\lambda Q\varphi} d\mu =  \sum_{i\in O_n} \hat\nu_n(i)\int e^{\lambda Q\varphi(\sigma t_{i,n})} d\mu_n(\sigma)\leq  \left( \sum_{i\in O_n} e^{\lambda R^{c(\ell)}\hat\varphi(i)} \hat\nu_n(i)\right) e^{K_{n-1}{\lambda^2}/8}\\
&\leq \exp\left[  \sum_{i=1}^n \hat\varphi(i) \hat\nu_n(i) + \frac{\lambda^2c(\ell)^2}{8}+ K_{n-1}c(\ell)^2\frac {\lambda^2}{8 }\right]=\exp\left[\lambda \int \varphi \,d\mu +K_{n}c(\ell)^2\frac {\lambda^2}8\right].
\end{align*}
This ends the proof of  \eqref{transn} for any $\mu\in {\mathcal M}_{{\mathcal T}_n}$.

The scheme of the induction proof of \eqref{transn}, with a better constant $c(\ell)$ when $\mu=\mu_o$ is the uniform measure on a $\ell$-local group $G$, is identical, starting from the second result of Lemma \ref{gateau} and using the property \eqref{inducmesbis}. This is left to the reader. 
\vspace{0,3 cm}

We now turn to the induction proof of the  dual formulation \eqref{symtilde} of the weak transport-entropy inequality \eqref{Ttilde}. The sketch of the proof is identical to the one of \eqref{transn}.

For the initial step $n=2$, one has  $G=S_2$ and $\ell=2$,  and one may easily check that 
\[\widetilde Q_{1}\varphi(\sigma)=\inf_{p\in \P(S_2)}\left\{\int \varphi dp + \frac1{2}\left( \int \1_{\sigma\neq \tau} \,dp(\tau)\right)^2\right\}.\] 
In that case, the result follows from the following  infimum-convolution property. 
\begin{lemma}\label{completegraph}  For any probability measure  $\nu$ on a Polish metric space  $\mathcal X$, for all  $\alpha\in (0,1)$ and all
 measurable functions  $f:\mathcal{X}\to \R $,  bounded from below
 \[\left(\int e^ {\alpha \widetilde R^\alpha f} d\nu\right)^{1/\alpha}\left(\int e^{-(1-\alpha)f} d\nu\right)^{1/(1-\alpha)}\leq 1,\]
 where for all $x\in \X$,  
 \[ \widetilde R^\alpha f ( x  )=\inf_{ p\in \P(\X)}  \left\{  \int f(y) d p(y)  +c_\alpha \left( \int \1_{x \neq y } d p(y)\right)^2 \right\},\]
  and  $c_\alpha$ is the  convex function defined by 
\[c_\alpha (u)=\frac{\alpha(1-u)\log(1-u)-(1-\alpha u)\log(1-\alpha u)}{\alpha(1-\alpha)},\quad u\in[0,1].\]
Observing that $c_\alpha (u)\geq u^2/2$ for all $u\in [0,1]$, the above inequality also holds replacing  $\widetilde R^\alpha f$ by  
 \begin{eqnarray}\label{deftildeR}
  \widetilde R f ( x  )=\inf_{ p\in \P(\X)}  \left\{  \int f(y) d p(y)  +\frac12 \left( \int \1_{x \neq y } d p(y)\right)^2 \right\},\qquad x\in\X.
  \end{eqnarray}
\end{lemma}

The proof of this Lemma can be found in \cite{Sam07} (inequality (4)). For a sake of completeness, we give in the Appendix a new proof of this result on  finite spaces $\X$ by using a localization argument (Lemma \ref{chose}). 

Let us now present the  key lemma for  the induction step of the proof.
For any function  $f:H_n\to \R$ and any $\sigma\in H_n$, we define 
\[\widetilde Q^{H_n}_t f(\sigma):= \inf_{p\in \P(H_n)}\left\{\int f \,dp +\frac 1 {2c(\ell)^2t}\left(\int d(\sigma,\tau) \,dp(\tau)\right)^2\right\}.\]
Here, writing $Q^{H_n}_t f$, we omit the dependence in $c(\ell)$ to simplify the notations. 
The proof relies on the following Lemma.
\begin{lemma}\label{lemmeQtilde} Let $i\in  O_n$. For any function $\varphi: H_i\to \R$ and any $\sigma\in H_n$, one has
\begin{enumerate}
\item\label{lemme23(1)} $\displaystyle \widetilde Q _{K_n}\varphi (\sigma t_{i,n})\leq \inf_{\hat p\in \P( O_n)}\left\{\sum_{l\in O_n} \widetilde Q^{H_n}_{K_{n-1}} \varphi^{t_{l,n}}(\sigma) \hat p(l) +\frac 1 2 \bigg(\sum_{l\in O_n} \1_{l\neq i} \hat p(l)\bigg)^2\right\},$\\
with  $c(\ell)=\min(2\ell-1,n) $ if $d=d_H$ and $c(\ell)=2$ if $d=d_T$ .

\item $\displaystyle \widetilde Q _{K_n}\varphi (\sigma t_{i,n})\leq \inf_{\hat p\in \P( O_n)}\left\{\sum_{l\in O_n} \widetilde Q^{H_n}_{K_{n-1}} \varphi^{t_{i,n}t_{i,l}^{-1}}(\sigma) \hat p(l) +\frac 1 2 \bigg(\sum_{l\in O_n}\1_{l\neq i} \hat p(l)\bigg)^2\right\},$\\
where $c(\ell)=\ell $ if $d=d_H$ and $c(\ell)=1$ if $d=d_T$, and $t_{i,l}$  denotes  an element of $G$ with $\degr(t_{i,l})\leq \ell$ and such that  $t_{i,l}(i)=l$.
\end{enumerate}
\end{lemma}
The proof of this lemma is similar to the one of Lemma \ref{gateau}. By   \eqref{triangular}  and the inequality 
\[(u+v)^2\leq \frac{u^2}{s}+\frac{v^2}{1-s}, \qquad  u,v\in \R, \quad  s\in(0,1),\]
we get for any $s\in(0,1)$,
\begin{align*}
\left(\int d(\sigma t_{l,n},\tau) \,dp(\tau)\right)^2 &\leq \left( \sum_{l\in  O_n} \int_{H_l} d(\sigma,\tau t_{l,n}^{-1}) dp_l(\tau) \hat p(l) +c(\ell) \sum_{l\in  O_n} \1_{l\neq i}   \hat p(l)\right)^2\\
&\leq \frac 1s\left( \sum_{l\in  O_n} \int_{H_l} d(\sigma,\tau t_{l,n}^{-1}) dp_l(\tau) \hat p(l) \right)^2+
\frac{c(\ell)^2}{1-s}\bigg(\sum_{l\in  O_n} \1_{l\neq i}   \hat p(l)\bigg)^2\\
&\leq \frac 1s  \sum_{l\in  O_n} \left(\int_{H_l} d(\sigma,\tau t_{l,n}^{-1})dp_l(\tau) \right)^2\hat p(l) +
\frac{c(\ell)^2}{1-s}\bigg(\sum_{l\in  O_n} \1_{l\neq i}   \hat p(l)\bigg)^2.
\end{align*}
 It follows that for any $\sigma\in H_n$,
\begin{align*}
&\widetilde Q_{K_n}\varphi(\sigma  t_{l,n})\\& \leq \inf_{\hat p\in \P( O_n)}\inf_{p_l\in \P(H_l),l\in O_n} \left\{
\sum_{l\in O_n}\left[ \int \varphi \,dp_l+\frac{1}{2c(\ell)^2sK_n} \left(\int_{H_l} d(\sigma,\tau t_{l,n}^{-1}) dp_l(\tau) \right)^2\right] \hat p(l) \right.\\
&\qquad\qquad\qquad\qquad\qquad\qquad\qquad\qquad\left.+\frac{1}{2(1-s)K_n}\bigg(\sum_{l\in O_n} \1_{l\neq i}   \hat p(l)\bigg)^2
\right\}\\
&= \inf_{\hat p\in \P( O_n)}\inf_{q_l\in \P(H_n),l\in O_n}
 \left\{
\sum_{l\in O_n}\left[ \int \varphi^{t_{l,n}} \,dq_l+\frac{1}{2c(\ell)^2sK_n} \left(\int_{H_n} d(\sigma,\tau) dq_l(\tau) \right)^2\right] \hat p(l) \right.\\
&\qquad\qquad\qquad\qquad\qquad\qquad\qquad\qquad\left.+\frac{1}{2(1-s)K_n}\bigg(\sum_{l\in O_n} \1_{l\neq i}   \hat p(l)\bigg)^2
\right\}\\
&=\inf_{\hat p\in \P( O_n)}\left\{\sum_{l\in O_n} \widetilde Q^{H_n}_{K_{n-1}} \varphi^{t_{l,n}}(\sigma) \hat p(l) +\frac 12 \bigg(\sum_{l\in O_n} \1_{l\neq i}   \hat p(l)\bigg)^2\right\},
\end{align*} 
where the last equality follows by choosing $s=K_{n-1}/K_n$, which ends the proof of the first inequality  of Lemma \ref{lemmeQtilde}. The second inequality of Lemma \ref{lemmeQtilde} is obtained identically starting from \eqref{triangularbis}.

We now turn to the induction step of the proof. By the decomposition of the measure $\mu$ on the $H_i$'s,  we want to bound
\[\int e^{\alpha \widetilde Q_{K_n} \varphi} d\mu=\sum_{i\in  O_n}\hat \nu_n(i) \int e^{\alpha \widetilde Q_{K_n} \varphi(\sigma)} d\mu_i(\sigma)= \sum_{i\in  O_n}\hat \nu_n(i) \int e^{\alpha \widetilde Q_{K_n} \varphi(\sigma t_{i,n})} d\mu_n(\sigma),\]
where the last equality is a consequence of property \eqref{inducmes}.

If $ O_n=\{n\}$,  then the result simply follows from the induction hypotheses applied to the measure $\mu_n$.

If  $ O_n\neq\{n\}$, then 
applying  successively  Lemma \ref{lemmeQtilde} \eqref{lemme23(1)}, the H\"older inequality, and the induction hypotheses, we get 
\begin{align*}
&\int e^{\alpha \widetilde Q_{K_n} \varphi(\sigma t_{i,n})} d\mu_n(\sigma)
\leq \inf_{\hat p\in \P( O_n)}
\left\{ \prod_{l\in O_n}\left(\int e^{\alpha \widetilde Q_{K_{n-1}}^{H_n} \varphi^{t_{l,n}}} d\mu_n \right)^{\hat p(l)}\exp\left[\frac 12 \bigg(\sum_{l\in O_n} \1_{l\neq i}   \hat p(l)\bigg)^2\right]\right\}\\
&\leq  \inf_{\hat p\in \P( O_n)}
\left\{ \prod_{l\in O_n}\left(\int e^{-(1-\alpha)   \varphi^{t_{l,n}}} d\mu_n \right)^{-\frac{\hat p(l)\alpha}{1-\alpha}}\exp\left[\frac 12 \bigg(\sum_{l\in O_n} \1_{l\neq i}   \hat p(l)\bigg)^2\right]\right\}\\
&= \exp \left[\alpha \inf_{\hat p\in \P( O_n)}\left\{\sum_{l\in O_n} \hat \varphi(l) \hat p(l)+\frac 12 \bigg(\sum_{l\in O_n} \1_{l\neq i}   \hat p(l)\bigg)^2 \right\} \right],
\end{align*}
where by property \eqref{inducmes}, we set
\[\hat \varphi(l):=\log  \left(\int e^{-(1-\alpha)   \varphi} d\mu_l \right)^{-\frac{1}{1-\alpha}}=\log  \left(\int e^{-(1-\alpha)   \varphi^{t_{l,n}}} d\mu_n \right)^{-\frac{1}{1-\alpha}}.\]
According to the definition of the infimum convolution $\widetilde R \hat \varphi$ on the space $\X=  O_n$ given in Lemma \ref{completegraph}, the last inequality is
\[\int e^{\alpha \widetilde Q_{K_n} \varphi(\sigma t_{i,n})} d\mu_n(\sigma)
\leq e^{\alpha \widetilde R \hat \varphi(i)},\]
and therefore Lemma \ref{completegraph}, applied with the measure $\nu=\hat \nu_n$, provides 
\begin{align*}
\int e^{\alpha \widetilde Q_{K_n} \varphi} d\mu &=  \sum_{i\in O_n} e^{\alpha  \widetilde R \hat \varphi(i)} \hat \nu_n(i) \leq  \bigg( \sum_{i\in O_n} e^{-(1-\alpha)\hat \varphi(i)}\hat \nu_n(i)\bigg)^{-\frac{\alpha}{1-\alpha}} \\
&= \bigg( \sum_{i\in O_n} \hat \nu_n(i)  \int e^{-(1-\alpha) \varphi} d\mu_i \bigg)^{-\frac{\alpha}{1-\alpha}}=\left( \int e^{-(1-\alpha) \varphi} d\mu\right)^{-\frac{\alpha}{1-\alpha}}.
\end{align*}
The proof of \eqref{symtilde} is completed for any measure $\mu\in \mathcal M$. To improve the constant when $\mu=\mu_o$ is the uniform law on a $\ell$-local group $G$, the proof is similar using the second inequality of Lemma \ref{lemmeQtilde} together with  property \eqref{inducmesbis}.
\end{proof}

  \begin{proof}[Proof of  (b) in Theorem \ref{transportsymetrique}] 
 We prove the dual equivalent property \eqref{syminf2}  as a consequence of the stronger following result: for any real function $\varphi$ on $G$, for any $j\in \{1, \ldots ,n\}$
\begin{eqnarray}\label{improvedthm}
\left(\int e^ {\alpha Q^j\varphi} d\mu\right)^{1/\alpha}\left(\int e^{-(1-\alpha)\varphi} d\mu\right)^{1/(1-\alpha)}\leq 1,
\end{eqnarray} 
where the infimum convolution operator $ Q^j\varphi$ is defined as follows, for $\sigma\in G$
 \begin{multline}  Q^j\varphi(\sigma)=\inf_{p_\in \P(G)} \left\{  \int \varphi dp + \frac 1{c(\ell)^2}\left(\int \1_{\sigma(j)\neq y(j)} dp(y)\right)^2 \right.\\
 \left.+\frac 1{2c(\ell)^2} \sum_{k\in[n]\setminus \{j\}} \left(\int \1_{\sigma(k)\neq y(k)} dp(y)\right)^2 \right\}\label{defQj}.
 \end{multline}

The proof of \eqref{improvedthm} relies on Lemma \ref{completegraph} and the  following ones.  For any $\sigma\in G$, we define
\[    Q^{H_n} \varphi(\sigma)=\inf_{p\in {\mathcal P}(H_n)}
 \left\{  \int \varphi dp +\frac 1{2c(\ell)^2} \sum_{k=1}^{n-1} \left(\int \1_{\sigma(k)\neq y(k)} dp(y)\right)^2 \right\},\]
 and for $j\in [n-1]$, 
 \begin{align*}
    Q^{H_n,j} \varphi(\sigma)&=\inf_{p\in {\mathcal P}(H_n)}
 \left\{  \int \varphi dp + \frac 1{c(\ell)^2}\left(\int \1_{\sigma(j)\neq y(j)} dp(y)\right)^2\right.\\
 &\left.\qquad\qquad\qquad  +\frac 1{2c(\ell)^2} \sum_{k\in[n-1] \setminus \{j\}} \left(\int \1_{\sigma(k)\neq y(k)} dp(y)\right)^2 \right\} . 
 \end{align*}

\begin{lemma}\label{inverse} Let $j\in [n]$. For any  $\sigma\in G$, one has 
\[ Q^j\varphi(\sigma)=Q^{\sigma(j)}\varphi^{\{-1\}}(\sigma^{-1}), \]
where $\varphi^{\{-1\}}(z)=\varphi(z^{-1}),z\in G$.
\end{lemma}
This result follows from the change of variables $\sigma(k)=l$   in the  definition \eqref{defQj} of $Q^j\varphi(\sigma)$, one has 
 \begin{align*}
 Q^j\varphi(\sigma)&=\inf_{p\in \P(G)} \left\{  \int \varphi dp +  \frac 1{c(\ell)^2}\left(\int \1_{y^{-1}(\sigma(j))\neq \sigma^{-1}(\sigma(j))} dp(y)\right)^2 \right.
\\
&\left.\qquad\qquad\qquad+\frac 1{2c(\ell)^2} \sum_{l, l\neq \sigma(j)} \left(\int \1_{l\neq y(\sigma^{-1}(l))} dp(y)\right)^2 \right\}\\
&=\inf_{q\in \P(G)} \left\{  \int \varphi (z^{-1})\, dq(z) + \frac 1{c(\ell)^2}\left(\int \1_{z(\sigma(j))\neq \sigma^{-1}(\sigma(j))} dq(z) \right)^2 \right.
\\
&\left.\qquad\qquad\qquad+\frac 1{2c(\ell)^2} \sum_{l, l\neq \sigma(j)}  \left(\int \1_{z(l)\neq \sigma^{-1}(l)} dq(z) \right)^2 \right\},\\
\end{align*}
where for the last equality,  we use the fact that the map that associates to any measure $p\in \P(G)$  the image measure $q:=R\#p$ with $R:\sigma\in G\mapsto \sigma^{-1}\in G$, is one to one from $\P(G)$ to $ \P(G)$.

Here is the key lemma for the induction step of the proof of  \eqref{improvedthm}.
\begin{lemma}\label{lemclef} \begin{enumerate}
\item Let $j\in   O_n$.
For any $\sigma \in H_n$, one has 
 \[Q^j\varphi(\sigma t_{j,n})\leq Q^{H_n} \varphi^{t_{j,n}}(\sigma).\]

 \item For any $\ell\geq 2$, let $c^2(\ell):= 8(\ell-1)^2 +2$. Assume that  $ O_n\neq \{n\}$ and let  $i,j\in   O_n$, $i\neq j$.
 We note  $D_i=\supp(t_{j,n}^{-1}t_{i,n})\setminus \{i\}$ and $d=|D_i|$.
 For any $\sigma \in H_n$,  for any $\theta\in [0,1]$  one has 
\begin{eqnarray*}
Q^i \varphi(\sigma t_{i,n})\leq \frac 1d \sum_{l\in t_{i,n}(D_i)} \left[\theta Q^{H_n,l} \varphi^{t_{i,n}}(\sigma)+(1-\theta) Q^{H_n} \varphi^{t_{j,n}}(\sigma)\right]+\frac12(1-\theta)^2.
\end{eqnarray*}

 \item For any $\ell\geq 2$, let $c^2(\ell):= 2(\ell-1)^2 +2$.
 Assume that  $ O_n\neq \{n\}$ and let  $i,j\in   O_n$, $i\neq j$.
Let  $t_{i,j}\in G$ such that $t_{i,j}(i)=j$ and $\degr(t_{i,j})\leq \ell$. We note  $D_i=\supp(t_{i,j})\setminus \{i\}$ and $d=|D_i|$.
 For any $\sigma \in H_n$,  for any $\theta\in [0,1]$  one has 
\begin{eqnarray*}
Q^i \varphi(\sigma t_{i,n})\leq \frac 1d \sum_{l\in t_{i,n}(D_i)} \left[\theta Q^{H_n,l} \varphi^{t_{i,n}}(\sigma)+(1-\theta) Q^{H_n} \varphi^{t_{i,n}t_{i,j}^{-1}}(\sigma)\right]+\frac12(1-\theta)^2.
\end{eqnarray*}
\end{enumerate}
\end{lemma}

\begin{proof} 
The first part of this Lemma follows from the fact that ${\mathcal P}(H_j)\subset {\mathcal P}(G)$  and the fact that $\int \1_{\sigma t_{j,n} (j)\neq y(j)} dp(y)=0$ for $\sigma\in H_n$ and  $p\in {\mathcal P}(H_j)$. Therefore, according to the definition of $Q^j\varphi$, one has for $\sigma\in H_j$, 
\begin{align*}
 & Q^j\varphi(\sigma t_{j,n})\leq \inf_{p\in {\mathcal P}(H_j)} \left\{  \int \varphi dp +\frac 1{2c(\ell)^2} \sum_{k\in [n]\setminus \{j\}}\left(\int \1_{\sigma  t_{j,n}(k)\neq y(k)} dp(y)\right)^2 \right\}\\
  &=\inf_{q\in {\mathcal P}(H_n)} \left\{  \int \varphi^{t_{j,n}} dq +\frac 1{2c(\ell)^2}\sum_{k\in [n]\setminus \{j\}}\left(\int \1_{\sigma t_{j,n}(k)\neq yt_{j,n}(k)} dq(y)\right)^2 \right\}
   =Q^{H_n}\varphi ^{t_{j,n}}(\sigma). 
   \end{align*}

For the proof of the second part of Lemma \ref{lemclef}, we set 
\[\tilde t_{i,j}:=t_{j,n}^{-1} t_{i,n}.\] Let us consider $p_i^l, l\in D_i$, a collection of measures in ${\mathcal P}(H_i)$, and $p_j\in {\mathcal P}(H_j)$ ($j\neq i$). For $\theta \in [0,1]$, \[p:=\frac1d \sum_{l\in D_i}[\theta p_i^l+(1-\theta)p_j],\] is a probability measure on $G$. Therefore, according to the definition of $Q^i\varphi$, for any $\sigma\in H_n$,  
\[Q^i\varphi(\sigma t_{i,n})\leq \frac 1d  \sum_{l\in D_i} \left[\theta \int f dp_i^l + (1-\theta )\int f dp_j\right]+ \frac 1{2c(\ell)^2}(A+B+C)  ,\]
with 
\[ A= \sum_{k\in [n]\setminus \supp(\tilde t_{i,j})} \left(\int \1_{\sigma t_{i,n}(k)\neq y(k)} dp(y)\right)^2, \quad B=\sum_{k\in D_i} \left(\int \1_{\sigma t_{i,n}(k)\neq y(k)} dp(y)\right)^2,\]
and  $\displaystyle C=2\left(\int \1_{\sigma t_{i,n}(i)\neq y(i)} dp(y)\right)^2$.

Since $\sigma\in H_n$ and $p_i^l\in\P(H_i)$, one has $\int \1_{\sigma t_{i,n}(i)\neq y(i)} dp_i^l(y)=0$ and $\int \1_{\sigma t_{i,n}(i)\neq y(i)} dp_j(y)=1$. It follows that  
\[C=2(1-\theta)^2.\]
For any $k\in [n]$ and $l\in D_i$,
let us note \[U_i(k,l):=\int \1_{\sigma t_{i,n}(k)\neq y(k)} dp_i^l(y),\quad {\mbox and }\quad U_j(k):=\int \1_{\sigma t_{i,n}(k)\neq y(k)} dp_j(y).\]
By the Cauchy-Schwarz inequality, one has 
\[A\leq  \frac 1d  \sum_{l\in D_i} \left[ \theta  \sum_{k\in [n]\setminus \supp(\tilde t_{i,j})} U_i^2(k,l) + (1-\theta) \sum_{k\in [n]\setminus \supp(\tilde t_{i,j})} U_j^2(k)\right].\]
We also have 
\begin{align*}
B&= \sum_{k\in D_i}\left(\frac \theta d U_i(k,k)+(1-\theta)U_j(k)+\frac\theta d\sum_{l\in D_i\setminus \{k\}}U_i(k,l)\right)^2\\
&\leq \sum_{k\in D_i} \left[d \left(\frac \theta d U_i(k,k)+(1-\theta)U_j(k)\right)^2+ \frac{\theta^2}  d\sum_{l\in D_i\setminus \{k\}}U_i^2(k,l)\right]\\
&\leq \sum_{k\in D_i} \left[ \frac {2\theta^2} d U_i^2(k,k)+2d(1-\theta)^2+ \frac{\theta^2}  d\sum_{l\in D_i\setminus \{k\}}U_i^2(k,l)\right]\\
&\leq  2d^2(1-\theta)^2 +\frac \theta d \sum_{l\in D_i}  \left[ 2U_i^2(l,l) + \sum_{k\in D_i\setminus \{l\}} U_i^2(k,l)\right]
\end{align*}
All the above estimates together provide 
\begin{multline*}
A+B+C\leq (2d^2+2) (1-\theta)^2\\ +\frac 1d  \sum_{l\in D_i} \left[ \theta \left( 2U_i^2(l,l) + \sum_{k\in [n]\setminus \{i,l\}} U_i^2(k,l)\right) + (1-\theta) \sum_{k\in [n]\setminus \supp(\tilde t_{i,j})} U_j^2(k)\right] . 
\end{multline*}
Observe that \[d=\degr(\tilde t_{i,j})-1=\degr(t_{j,n}^{-1} t_{i,n})-1\leq 2\ell-2.\] Therefore, according to the definition of $c(\ell)$,  one has 
$2d^2+2\leq c(\ell)^2.$ As a consequence  we get from all estimates above, by optimizing over all $p_i^l\in \P(H_i)$ and all $p_j\in \P(H_j)$,
\[Q^i\varphi(\sigma t_{i,n})\leq \frac 1d  \sum_{l\in D_i} \left[\theta V_l + (1-\theta )W_j\right]+ \frac 12 (1-\theta)^2  ,\]
with 
\begin{align*}
V_l&:= \inf_{p_i\in\P(H_i)}\left\{\int \varphi dp_i+\frac 1{ c(\ell)^2}\left(\int \1_{\sigma t_{i,n}(l)\neq y(l)} dp_i(y)\right)^2  \right.\\
&\qquad\qquad\qquad\qquad\qquad\qquad\qquad\qquad\left.+ \frac 1{ 2c(\ell)^2}\sum_{k\in [n]\setminus \{i,l\}}
\left(\int\1_{\sigma t_{i,n}(k)\neq y(k)} dp_i(y)\right)^2\right\}\\
&=\inf_{q_i\in\P(H_n)}\left\{\int \varphi^{t_{i,n}} dq_i+\frac 1{ c(\ell)^2}\left(\int \1_{\sigma (t_{i,n}(l))\neq y(t_{i,n}(l))} dq_i(y)\right)^2 \right.\\
&\qquad\qquad\qquad\qquad\qquad\qquad\qquad\qquad\left.+ \frac1{ 2c(\ell)^2}\sum_{k\in [n-1]\setminus \{t_{i,n}(l)\}}
\left(\int\1_{\sigma (k)\neq y(k)} dq_i(y)\right)^2\right\}\\
&=Q^{H_n, t_{i,n}(l) }\varphi^{t_{i,n}}(\sigma) 
\end{align*}
and 
\begin{align*}
W_j&:=\inf_{p_j\in\P(H_j)}\left\{\int \varphi dp_j+\frac 1{2 c(\ell)^2} \sum_{k\in [n]\setminus \supp(\tilde t_{i,j})}
\left(\int \1_{\sigma t_{i,n}(k)\neq y(k)} dp_j(y)\right)^2\right\}\\
&= \inf_{q_j\in\P(H_n)}\left\{\int \varphi^{t_{j,n}} dq_j+\frac 1{2 c(\ell)^2} \sum_{k\in [n]\setminus  \supp( t_{j,n}^{-1}t_{i,n})}
\left(\int\1_{\sigma t_{i,n}(k)\neq y t_{j,n}(k)} dq_j(y)\right)^2\right\}\\
&\leq  \inf_{q_j\in\P(H_n)}\left\{\int \varphi^{t_{j,n}} dq_j+\frac 1{2 c(\ell)^2} \sum_{k\in [n]\setminus  \{i\}}
\left(\int\1_{\sigma t_{i,n}(k)\neq y t_{i,n}(k)} dq_j(y)\right)^2\right\}\\
&=\inf_{q_j\in\P(H_n)}\left\{\int \varphi^{t_{j,n}} dq_j+\frac 1{2 c(\ell)^2} \sum_{k\in [n-1]}
\left(\int\1_{\sigma (k)\neq y (k)} dq_j(y)\right)^2\right\}\\
&=Q^{H_n}\varphi^{t_{j,n}}(\sigma)
\end{align*}
where we used successively the following arguments:  $H_n t_{j,n}= H_j$;  if $k\in [n]\setminus  \supp( t_{j,n}^{-1}t_{i,n})$ then $t_{i,n}(k)=t_{j,n}(k)$; $[n]\setminus  \supp( t_{j,n}^{-1}t_{i,n})\subset [n]\setminus\{i\}$.
This ends the proof of part $(2)$ of Lemma \ref{lemclef}.

The proof of part (3) Lemma \ref{lemclef} is identical replacing $\tilde t_{i,j}$ by $ t_{i,j}$. In that case 
\[2 d^2+2\leq 2(\ell-1)^2 +2=c^2(\ell).\]
Then, the only minor change is for the last step 
\begin{align*}
W_j&:=\inf_{p_j\in\P(H_j)}\left\{\int \varphi dp_j+\frac 1{2 c(\ell)^2} \sum_{k\in [n]\setminus \supp( t_{i,j})}
\left(\int \1_{\sigma t_{i,n}(k)\neq y(k)} dp_j(y)\right)^2\right\}\\
&= \inf_{q_j\in\P(H_n)}\left\{\int \varphi^{t_{i,n}t_{i,j}^{-1}} dq_j+\frac 1{2 c(\ell)^2} \sum_{k\in [n]\setminus  \supp( t_{i,j})}
\left(\int\1_{\sigma t_{i,n}(k)\neq y t_{i,n}t_{i,j}^{-1}(k)} dq_j(y)\right)^2\right\}\\
&\leq  \inf_{q_j\in\P(H_n)}\left\{\int \varphi^{t_{i,n}t_{i,j}^{-1}} dq_j+\frac 1{2 c(\ell)^2} \sum_{k\in [n]\setminus  \{i\}}
\left(\int\1_{\sigma t_{i,n}(k)\neq y t_{i,n}(k)} dq_j(y)\right)^2\right\}\\
&=\inf_{q_j\in\P(H_n)}\left\{\int \varphi^{t_{i,n}t_{i,j}^{-1}} dq_j+\frac 1{2 c(\ell)^2} \sum_{k\in [n-1]}
\left(\int\1_{\sigma (k)\neq y (k)} dq_j(y)\right)^2\right\}\\
&=Q^{H_n}\varphi^{t_{i,n}t_{i,j}^{-1}}(\sigma)
\end{align*}
where we used successively the following arguments:  $H_n t_{i,n}t_{i,j}^{-1}= H_j$;  if $k\in [n]\setminus  \supp( t_{i,j})$ then $t_{i,j}(k)=k$; $[n]\setminus  \supp( t_{i,j})\subset [n]\setminus\{i\}$.
The proof of Lemma \ref{lemclef} is completed.
\end{proof}

 We will now prove \eqref{improvedthm} by induction over $n$.
For $n=2$,  $G$ is  the two points space $S_2$ which is $2$-local.  For $i\in\{1,2\}$, and for any  $p\in\P(G)$, 
\begin{multline*}
 \frac 1{c(2)^2}\left(\int \1_{\sigma(i)\neq y(i)} dp(y)\right)^2+\frac 1{2c(2)^2} \sum_{k,k\neq i}  \left(\int \1_{\sigma(k)\neq y(k)} dp(y)\right)^2 \\= \frac 38  \left(\int \1_{\sigma \neq y } dp(y)\right)^2 \leq \frac1 2\left(\int \1_{\sigma \neq y } dp(y)\right)^2 .\end{multline*} 
As a consequence,  we get the expected result from Lemma \ref{completegraph} applied with $\X=G$.

We will  now present the induction step. We assume that \eqref{improvedthm} holds at the rank $n-1$ for all $j\in\{1,\ldots, n-1\}$. 

Let us first explain that  it suffices to prove  
\eqref{improvedthm}  for $j=n$. 
For any $t\in S_n$, let $G^{(t)}=t^{-1}G t$. The  isomorphism $c_t:G\to G^{(t)}$, $\sigma\mapsto t^{-1}\sigma t$    pushes forward the measure $\mu$ on the measure $\mu^{(t)}:=c_t\# \mu \in{\mathcal P}(G^{(t)})$, and conversely $\mu=c_{t^{-1}}\# \mu^{(t)}$.
Let $j\in [n]$. For any $\sigma\in G^{(t)}$ and any real function $\varphi$ on $G$, one  has 
\begin{align*}
 (Q^j\varphi)\circ c_{t^{-1}}(\sigma)&=\inf_{p\in\P(G)}\left\{\int \varphi\, dp +\frac1{ c(\ell)^2} \left(\int \1_{t\sigma t^{-1}(j)\neq y(j)} \,dp(y)\right)^2\right.\\
 &\qquad\qquad\qquad\qquad\left.+\frac1{2 c(\ell)^2}\sum_{k\in[n]} \left(\int \1_{t\sigma t^{-1}(k)\neq y(k)} \,dp(y)\right)^2\right\}\\
 &=\inf_{q\in\P(G^{(t)})}\left\{\int \varphi\circ c_{t^{-1}} \,dq +\frac1{ c(\ell)^2} \left(\int \1_{t\sigma t^{-1}(j)\neq ty t^{-1}(j)} \,dq(y)\right)^2\right.\\
 &\qquad\qquad\qquad\qquad\left.+\frac1{2 c(\ell)^2}\sum_{k\in[n]} \left(\int \1_{t\sigma t^{-1}(k)\neq ty t^{-1}(k)} \,dq(y)\right)^2\right\}\\
 &=\inf_{q\in\P(G^{(t)})}\left\{\int \varphi\circ c_{t^{-1}} \,dq +\frac1{ c(\ell)^2} \left(\int \1_{\sigma t^{-1}(j)\neq y t^{-1}(j)}\, dq(y)\right)^2\right.\\
 &\qquad\qquad\qquad\qquad\left.+\frac1{2 c(\ell)^2}\sum_{k\in[n]} \left(\int \1_{\sigma (k)\neq y (k)} \,dq(y)\right)^2\right\}\\
 &=Q^{t^{-1}(j)}(\varphi\circ c_{t^{-1}})(\sigma).
 \end{align*}
 From this observation, by choosing $t^{-1}=t_{jn}$, and setting $\psi= \varphi\circ c_{t^{-1}}$, one has 
 \begin{align*}
 &\left(\int_G e^ {\alpha Q^j\varphi} d\mu\right)^{1/\alpha}\left(\int_G e^{-(1-\alpha)\varphi} d\mu\right)^{1/(1-\alpha)}\\&= \left(\int_{G^{(t)}} e^ {\alpha (Q^j\varphi)\circ c_{t^{-1}}} d\mu^{(t)}\right)^{1/\alpha}\left(\int_{G^{(t)}}e^{-(1-\alpha)\varphi\circ c_{t^{-1}}} d\mu^{(t)}\right)^{1/(1-\alpha)}\\
 &=\left(\int_{G^{(t)}} e^ {\alpha Q^n \psi} d\mu^{(t)}\right)^{1/\alpha}\left(\int_{G^{(t)}} e^{-(1-\alpha)\psi} d\mu^{(t)}\right)^{1/(1-\alpha)}
 \end{align*}

If  we assume that $G$ is a normal subgroup of $S_n$ and that $\mu$ satisfies the second property of \eqref{hypoinv}, then $G^{(t)}=G$ and $\mu^{(t)}=\mu$. Therefore the above expression is bounded by 1 as soon as \eqref{improvedthm} holds for $j=n$.
If we assume $G$ is a $\ell$-local group and $\mu=\mu_o$ is the uniform law on $G$, then 
$G^{(t)}$ is also a $\ell$-local group and $\mu^{(t)}$ is exactly the uniform law on $G^{(t)}$. Therefore the last expression is bounded by 1 as soon as \eqref{improvedthm} holds with  $j=n$  for any uniform law on a $\ell$-local group.
As a conclusion,  it remains  to prove inequality  
\eqref{improvedthm}  for $j=n$.


 We may assume that $ O_n\neq \{n\}$, otherwise the induction step is obvious.
We first apply  Lemma \ref{inverse}, by the first property of \eqref{hypoinv} satisfied by $\mu$, 
 \[ \int e^{\alpha Q^n\varphi}d\mu =  \int e^{\alpha Q^{\sigma(n)}\varphi^{\{-1\}}(\sigma^{-1})}d\mu(\sigma)= \int e^{\alpha Q^{\sigma^{-1}(n)}\varphi^{\{-1\}}(\sigma)}d\mu(\sigma).\]
 Let $g=\varphi^{\{-1\}}$. According to the decomposition of the measure $\mu$ on the sets $H_i, i\in  O_n$, 
  \begin{eqnarray}\label{gateaueau} \int e^{\alpha Q^n\varphi}d\mu =  \sum_{i\in O_n} \hat \nu_n(i)\int e^{\alpha Q^ig} d\mu_i.
  \end{eqnarray}
  For $k\in O_n$, let us note
  \[\hat g (k):=\log \left(\int e^{-(1-\alpha)g} d\mu_k\right)^{-1/(1-\alpha)}.\]
  We choose $j\in O_n$ such that 
  \[\min_{k\in O_n} \hat g (k) =  \hat g (j).\]
 
 By property \eqref{inducmes} and then applying Lemma \ref{lemclef} (1), we get  
 \[\int e^{\alpha Q^j g} d\mu_j=\int e^{\alpha Q^j g(\sigma t_{i,n})} d\mu_n(\sigma) \leq \int e^{\alpha Q^{H_n} g^{t_{j,n}}} d\mu_n.\]
 By the induction hypotheses applied to the measure $\mu_n$ on the subgroup $H_n=G_{n-1}$, it follows that 
  \begin{align}
\int e^{\alpha Q^jg} d\mu_j &\leq \left(\int e^{-(1-\alpha) g^{ t_{jn}}} \,d\mu_n\right)^{-\alpha/(1-\alpha)} \nonumber \\ 
&=\left(\int e^{-(1-\alpha) g} \,d\mu_j\right)^{-\alpha/(1-\alpha)}=e^{\alpha  \hat g (j)}.\label{tarte}
\end{align}
Let us now consider $i\neq j$,  $i\in O_n$. When $G$ is a normal subgroup of $S_n$, property \eqref{inducmes}, the second part of  Lemma \ref{lemclef}  and Jensen's inequality yield: for any $\theta\in [0,1]$, 
\begin{align*}
&\int e^{\alpha Q^ig} d\mu_i=\int e^{\alpha Q^i g(\sigma t_{i,n})} d\mu_n(\sigma)\\
&\leq \exp\left\{\frac 1d \sum_{l\in t_{i,n}(D_i)}\left[ 
\theta \log \int e^{\alpha Q^{H_n,l} g^{t_{i,n}}} d\mu_n + (1-\theta) \log \int e^{\alpha Q^{H_n}g^{t_{j,n}}} d\mu_n \right]+\frac{\alpha}2(1-\theta)^2\right\}\\
\end{align*}
By the induction hypotheses applied with the  measure $\mu_n$ on the normal  subgroup $G_{n-1}=H_n$ of $S_{n-1}$, and from property \eqref{inducmes}, it follows that 
 \begin{eqnarray}\label{pomme}
\qquad\int e^{\alpha Q^ig} d\mu_i\leq \exp\left\{\theta \alpha  \hat g (i) + (1-\theta)\alpha \hat g (j) +\frac{\alpha}{2}(1-\theta)^2\right\}.
\end{eqnarray}
We get the same inequality when $G$ is a $\ell$-local group and $\mu=\mu_o$ is the uniform law on $G$, by using property \eqref{inducmesbis}, the third part of  Lemma \ref{lemclef}  and the induction hypotheses applied to the uniform measure $\mu_n$ on the $\ell$-local subgroup $G_{n-1}=H_n$.

According to the definition \eqref{deftildeR} of the infimum-convolution operator $\widetilde R \hat g$ defined on the space $\X= O_n$, we may easily check  that for every $i\in  O_n$, 
  \[ \widetilde R \hat g(  i )=\inf_{\theta\in[0,1]}  \left\{\theta \hat g (i) +(1-\theta)\min_{k\in O_n} \hat g (k)+\frac12 (1-\theta)^2 \right\}.\]
Therefore optimizing over all $\theta \in [0,1]$, we get  from \eqref{tarte} and \eqref{pomme}: for all $i\in  O_n$,
\[\int e^{\alpha Q^ig} d\mu_i\leq e^{\alpha \widetilde R \hat g (i)}.\]
Finally,  from Lemma \ref{completegraph} applied with   the  measure $\nu=\hat \nu_n$  on $ O_n$, the equality \eqref{gateaueau} gives
\begin{multline*}
\int e^{\alpha Q^n\varphi }d\mu \leq \int e^{\alpha \widetilde R \hat g} \,d\hat\nu_n\leq \left(\int e^{-(1-\alpha) \hat g} \,d\hat \nu_n\right)^{-\alpha/(1-\alpha)}=\left(\sum_{i\in O_n} \hat \nu_n(i) \int e^{-(1-\alpha)  g} \,d \mu_i\right)^{-\alpha/(1-\alpha)}\\= \left(\int e^{-(1-\alpha)  g} \,d \mu\right)^{-\alpha/(1-\alpha)}=\left(\int e^{-(1-\alpha)  \varphi} \,d \mu\right)^{-\alpha/(1-\alpha)}.
\end{multline*}
The proof of \eqref{improvedthm} is completed.
\end{proof}

\section{Transport-entropy inequalities  on the slice of the cube.}\label{sectionslice}

\begin{proof}[Proof of  (a) in Theorem \ref{transportslice}] We adapt to the space $\X_{k,n-k}$ the proof of  (a) in Theorem \ref{transportsymetrique}. 
In order to  avoid redundancy, we only present the main steps of the proof.

By duality, it suffices to prove that  for all  functions  $\varphi$ on  $\X_{k,n-k}$ and all  $\lambda\geq 0$,
\begin{eqnarray}\label{transnslice}
\int e^{\lambda  Q\varphi} d\mu_{k,n-k}\leq e^{\int \lambda \varphi \,d\mu_{k,n-k} +C_{k,n-k}\lambda^2/2},
\end{eqnarray}
where
\[Q\varphi(x)=\inf_{p\in \P(\X_{k,n-k})}\left\{\int \varphi dp +\int d_h(x,y) \,dp(x)\right\},\qquad x\in \X_{k,n-k},
\]
  and for any 
  $0< \alpha<1$, 
\begin{eqnarray}\label{slicetilde}
\left(\int e^ {\alpha  \widetilde Q_{C_{k,n-k}}\varphi} d\mu\right)^{1/\alpha}\left(\int e^{-(1-\alpha)\varphi} d\mu\right)^{1/(1-\alpha)}\leq 1,
\end{eqnarray}
where  for  $t> 0$,  
\[  \widetilde Q_t\varphi(x)=\inf_{p\in \P(\X_{k,n-k})} \left\{  \int \varphi \,dp  +\frac 1{2t}  \left(\int d_h(x,y) \,dp(y)\right)^2 \right\}, \quad x\in \X_{k,n-k}.\]

The proof is   by induction  over $n$ and $0\leq k\leq n$. 

For any $n\geq 1$, if $k=n$ or $k=0$, the set $\X_{k,n-k}$ is reduced to a singleton and the inequalities \eqref{transnslice} or \eqref{slicetilde} are obvious. 

For $n=2$ and $k=1$, $\X_{k,n-k}$ is a two points set,  \eqref{transnslice} and  \eqref{slicetilde} directly follows from property \eqref{Pinskerdual} and Lemma \ref{completegraph} on $\X=\X_{1,1}$.

For the induction step, we consider the collection of subset $\Omega_{i,j}$,
with $i,j\in\{1,\ldots, n\}, i\neq j$, defined by 
\[\Omega^{i,j}:=\left\{ x\in\X_{k,n-k}, x_i=0,x_j=1\right\}.\]
Since for any $x \in\X_{k,n-k}$,  
\[\sum_{(i,j),i\neq j} \1_{\Omega^{i,j}}(x)=k(n-k),\]
 any  probability measure
 $p$  on  $\X_{k,n-k}$ admits a unique decomposition  defined by 
\[p=\sum_{(i,j),i\neq j} \hat p(i,j) p^{i,j},  \qquad \mbox{ with  }\quad p^{i,j}=\frac{\1_{\Omega^{i,j}}p}{p(\Omega^{i,j})} \qquad \mbox{ and } \quad  \hat p(i,j)=\frac{p(\Omega^{i,j})}{k(n-k)}.\]
Thus, we define probability measures $p^{i,j}\in \P(\Omega^{i,j})$ and a  probability measure $\hat p$ on the set $I(n)=\{(i,j)\in\{1,\ldots,n\}^2, i\neq j\}$.
For the uniform law  $\mu$ on $\X_{k,n-k}$,  one has  
\[\mu=\frac 1{n(n-1)} \sum_{(i,j)\in I(n)}\mu^{i,j},\]
where $\mu^{i,j}$ is the uniform law on $\Omega^{i,j}$, $\mu_{i,j}(x)=\binom{n-2}{k-1}$, for any $x\in \Omega^{i,j}$.

For any $(i,j),(l,m)\in I(n)$, let  $s_{(i,j),(l,m)}:\X_{k,n-k}\to\X_{k,n-k}$ denote  the map that exchanges the coordinates $x_i$ by $x_l$  and $x_j$ by $x_m$ for any point $x\in \X_{k,n-k}$. This map is one to one from $\Omega^{i,j}$ to $\Omega^{l,m}$. 
For any $(i,j)\in I(n)$, the set $\Omega^                                                                                                                                                                                                                                                                                                                                                                                                                                                                                                                                                                                                                                                                                                                                                                                                                                                                               {i,j}$ can be identify to $\X_{k-1,n-k-1}$ and therefore the induction hypotheses apply for the uniform law  $\mu^{i,j}$ on $\Omega^{i,j}$ with Hamming distance
\[d^{i,j}_h(x,y)=\frac12 \sum_{k\in[n]\setminus\{i,j\}} \1_{x_k\neq y_k},\qquad x,y\in \Omega^{i,j}.\]

For any function  $f:\Omega^{i,j}\to \R$ and any $x\in \Omega^{i,j}$, we define 
\[Q^{\Omega^{i,j}} f(x):= \inf_{p\in \P(\Omega^{i,j})}\left\{\int f \,dp +\int d^{i,j}_h(x,y) \,dp(y)\right\},
\]
and
\[\widetilde Q^{\Omega^{i,j}}_t f(x):= \inf_{p\in \P(H_n)}\left\{\int f \,dp +\frac 1 {2t}\left(\int d_h^{i,j}(x,y) \,dp(x)\right)^2\right\}.\]

The key lemma of the proof that replaces Lemma \ref{gateau} and \ref{lemmeQtilde} is the following.

\begin{lemma}\label{lemslice} For any function $\varphi: \Omega^{i,j}\to \R$ and any $x\in \Omega^{i,j}$, one has
\[Q\varphi (x)\leq \inf_{\hat p\in \P(I(n))}\left\{\sum_{(l,m)\in I(n)} Q^{\Omega^{i,j}} (\varphi\circ s_{(i,j),(l,m)})(x) \hat p(l,m) +\sum_{(l,m)\in I(n)} \1_{(l,m)\neq (i,j)} \hat p(l,m)\right\},\]
and 
\begin{multline*} 
\widetilde Q _{C_{k,n-k}}\varphi (x)\leq \inf_{\hat p\in \P(I(n))}\left\{\sum_{(l,m)\in I(n)} \widetilde Q^{\Omega^{i,j}}_{C_{k-1,n-k-1}} (\varphi\circ s_{(i,j),(l,m)})(x) \hat p(l,m) \right.\\\left.
+\frac 1 2 \bigg(\sum_{(l,m)\in I(n)}\1_{(l,m)\neq (i,j)} \hat p(l,m)\bigg)^2\right\}
.
\end{multline*}
\end{lemma}
The proof of this lemma is obtained by decomposition of the measures $p\in \P(\X_{k,n-k})$ on the sets $\Omega^{i,j}$, and using the following  inequality\[d_h(x,y)\leq d_h^{i,j}(x,  s_{(i,j),(l,m)})(y))+ d_h(s_{(i,j),(l,m)})(y),y)\leq d_h^{i,j}(x,  s_{(i,j),(l,m)})(y))+2,\]
for any
$x\in \Omega^{i,j}$, $y\in \Omega^{l,m}$.

Finally, the proof of the induction step based on  Lemma \ref{lemslice} and the identity $C_{k,n-k}=C_{k-1,n-k-1}+1$, is left to the reader. 
\end{proof}
\begin{proof}[Proof of (b) in Theorem \ref{transportslice}] We will explain the projection argument on the dual formulations of the transport-entropy inequalities. 
According to  Proposition 4.5 and Theorem 9.5 of \cite{GRST16}, the weak transport-entropy inequality \eqref{That} is equivalent to the following property that we want to establish: for any real function $f$ on $\X_{k,n-k}$  and for any 
  $0< \alpha<1$, 
\begin{eqnarray}\label{noix}
\left(\int e^ {\alpha  \widehat Qf} d\mu_{k,n-k}\right)^{1/\alpha}\left(\int e^{-(1-\alpha)f} d\mu_{k,n-k}\right)^{1/(1-\alpha)}\leq 1,
\end{eqnarray}
where
\[  \widehat Qf(x):=\inf_{p\in \P(\X_{k,n-k}} \left\{  \int \varphi\,dp  +\frac 1{8} \sum_{k=1}^n \left(\int \1_{x_k\neq y_k} \,dp(y)\right)^2 \right\}, \qquad x\in \X_{k,n-k}.\]

Let us apply property \eqref{syminf2} to the function $f\circ P:S_n\to \R$. Since $\mu_{k,n-k}=P\#\mu$, we get 
\begin{eqnarray*}
\left(\int e^ {\alpha  \wideparen Q(f\circ P)} d\mu\right)^{1/\alpha}\left(\int e^{-(1-\alpha)f} d\mu_{k,n-k}\right)^{1/(1-\alpha)}\leq 1.
\end{eqnarray*}
The  inequality \eqref{noix} is an easy consequence of the following result.
\begin{lemma}\label{tartine} For any $\sigma\in S_n$, 
$\wideparen Q(f\circ P)(\sigma) \geq \widehat Qf(P(\sigma)).$ 
\end{lemma}
It remains to prove this lemma. By definition, one has 
\begin{align*}
\wideparen Q(f\circ P)(\sigma)&=\inf_{p\in \P(S_n)}\left\{\int f\circ P \, dp + \sum_{j=1}^n \left(\int \1_{\sigma(j)\neq \tau(j)} dp(\tau)\right)^2\right\}\\
&=\inf _{q\in \P(\X_{k,n-k})} \inf_{p\in S_n,P\#p=q}\left\{\int f\circ P \, dp + \sum_{j=1}^n \left(\int \1_{\sigma(j)\neq \tau(j)} dp(\tau)\right)^2\right\}\\
&=\inf _{q\in \P(\X_{k,n-k})}\left\{ \int f \,dq + \inf_{p\in S_n,P\#p=q} \left[\sum_{j=1}^n \left(\int \1_{\sigma(j)\neq \tau(j)} dp(\tau)\right)^2\right]\right\}.
\end{align*}
Let $p\in S_n$ such that $P\#p=q$. 

\[\int \1_{\sigma(j)\neq \tau(j)} dp(\tau)=\sum_{y\in\X_{k,n-k}} \sum_{\tau\in S_n}\1_{ P(\tau)=y,\sigma(j)\neq \tau(j)} p(\tau).\]
For  $y\in \X_{k,n-k}$, let us note $Y=\{i\in [n], y_i=1\}$. Then $P(\tau)=y$ if and only if 
$\tau([k])=Y$.

Assume that $j\in[k]$, if $\tau([k])=Y$ and $\sigma(j)\not\in Y$ then $\tau(j)\neq \sigma(j)$. Therefore one has 
\[\big\{\tau, \tau([k])=Y, \sigma(j)\not\in Y \big\}\subset \big\{\tau, P(\tau)=y,\sigma(j)\neq \tau(j)\big\}.\]
Assume now that $j\not \in[k]$, if $\tau([k])=Y$ and $\sigma(j)\in Y$ then we also have $\tau(j)\neq \sigma(j)$. It follows that  
\[\big\{\tau, \tau([k])=Y, \sigma(j)\in Y \big\}\subset \big\{\tau, P(\tau)=y,\sigma(j)\neq \tau(j)\big\}.\]
From these observations, we get 
\begin{align*}
\sum_{j=1}^n \left(\int \1_{\sigma(j)\neq \tau(j)} dp(\tau)\right)^2
&\geq \sum_{j\in [k]}  \left(\int \1_{P(\tau)=y,\sigma(j)\not\in Y} dp(\tau)\right)^2+ \sum_{j\in [n]\setminus[k]}  \left(\int \1_{P(\tau)=y,\sigma(j)\in Y} dp(\tau)\right)^2\\
&= \sum_{j\in [k]}  \left(\int \1_{\sigma(j)\not\in Y} dq(y)\right)^2+ \sum_{j\in [n]\setminus[k]}  \left(\int \1_{\sigma(j)\in Y} dq(y)\right)^2\\
&= \sum_{i\in \sigma([k])}  \left(\int \1_{i\not\in Y} dq(y)\right)^2+ \sum_{i\not\in \sigma([k])}  \left(\int \1_{i\in Y} dq(y)\right)^2\\
&= \sum_{i\in \sigma([k])}  \left(\int \1_{y_i=0} dq(y)\right)^2+ \sum_{i\not\in \sigma([k])}  \left(\int \1_{y_i=1} dq(y)\right)^2\\
\end{align*}
Setting  $x=P(\sigma)$, it follows that 
\begin{align*}
\sum_{j=1}^n \left(\int \1_{\sigma(j)\neq \tau(j)} dp(\tau)\right)^2
&\geq \sum_{i=1}^n  \left[ \1_{x_i=1} \left(\int \1_{y_i=0} dq(y)\right)^2+  \1_{x_i=0} \left(\int \1_{y_i=1} dq(y)\right)^2\right]\\
&=\sum_{i=1}^n \left(\int \1_{y_i\neq x_i} dq(y)\right)^2.
\end{align*}
This inequality provides
\[\wideparen Q(f\circ P)(\sigma)\geq \widehat Q f(x) =\widehat Q f(P(\sigma)).\]
The proof of Lemma \ref{tartine} and  (b) in Theorem \ref{transportslice} is completed.
\end{proof}

\section{Appendix}
\begin{proof}[Proof of Lemma  \ref{map}]
Let $\mathcal {T}=(t_{i_j,j})$ be a $\ell$-local base of a group of permutations $G=G_n$. In order to prove that the map 
\begin{eqnarray}\label{TT}\begin{array}{cccl} & O_2\times O_3\times \cdots \times O_n& \to &G\\
U_{\mathcal T}: &i_2,i_3, \dots ,i_n&\mapsto & t_{i_2,2} t_{i_3,3} \cdots t_{i_n,n}, \end{array}
\end{eqnarray}
is one to one, it suffises to construct its inverse. 

For any $j\in\{2,\ldots,n\}$, let $U_j$ denotes the map defined by  \[U_j(i_2,i_3, \dots ,i_j)=t_{i_2,2} t_{i_3,3} \cdots t_{i_j,j}.\]
Let $\sigma=\sigma^{(n)}\in G$.   We want to find the unique vector $(i_1,\ldots, i_n)\in O_1\times \cdots \times O_n$ such that \[U_n(i_1,\ldots, i_n)= U_{\mathcal T}(i_1,\ldots, i_n)=\sigma.\]
 Since $U_n(i_1,\ldots, i_n)(i_n)=n$, necessarily, one has to fixe $i_n=(\sigma^{(n)})^{-1}(n)$. $i_n$ belongs to  $O_n$. Let $\sigma^{(n-1)}=\sigma^{(n)}t_{(\sigma^{(n)})^{-1}(n),n}^{-1}$. On has $\sigma^{(n-1)}\in G_{n-1}$. 
Then, since \[n-1=U_{n-1}(i_1,\ldots,i_{n-1})(i_{n-1})=\sigma^{(n-1)}(i_{n-1}),\] we necessarily have 
$i_{n-1}= (\sigma^{(n-1)})^{-1}(n-1)\in O_{n-1}$. We set  \[\sigma^{(n-2)}=\sigma^{(n-1)}t_{(\sigma^{(n-1)})^{-1}(n-1),n-1}^{-1}\in G_{n-2}.\]
Following this induction procedure, we construct a family of permutations $\sigma^{(j)}\in G_{j}$ for $j\in [n]$, such that $i_j=(\sigma^{(j)})^{-1}(j)\in O_{j}$ for all $j\in\{2,\ldots,n\}$. 
Observing that $G_1=\{Id\}$, it follows that $\sigma^{(1)}=id$ and therefore 
\[\sigma=\sigma^{(n)}=t_{i_2,2} t_{i_3,3} \cdots t_{i_n,n}.\]
This ends the proof of Lemma  \ref{map}.
\end{proof}

\begin{proof}[Proof of Lemma  \ref{lembase}] Let $G=G_n$ be a $\ell$-local group. From the definition of the $\ell$-local property, it is clear that any of the subgroup $G_j$, $j\in \{2,\ldots,n\}$ is  $\ell$-local. As a consequence, for any $i_j\in O_j$, $i_j\neq j$, there exists $t_{i_j,j}\in G_j$ such that 
\[t_{i_j,j}(i_j)=j,\quad \mbox{and}\quad \degr(t_{i_jj})\leq \ell.\]
This completes the proof of Lemma  \ref{lembase}.
\end{proof}

\begin{proof}[Proof of Lemma  \ref{completegraph}]
Let $\alpha\in (0,1)$ and $f$ be a real  function on the finite set $\X$.
 We want to show that for any probability measure $\nu$ on $\X$, 
\[\left(\int e^ {\alpha \widetilde R^\alpha f} d\nu\right)^{1/\alpha}\left(\int e^{-(1-\alpha)h} d\nu\right)^{1/(1-\alpha)}\leq 1.\]
We will apply the following lemma whose proof is given at the end of this section.
\begin{lemma}\label{chose} Let $F$ be a real function  on $\X$ and $K\in \R$. Let us consider the set 
\[{\mathcal C}:=\left\{ \nu\in {\mathcal P}(\X), \, \int F  \,d\nu=K\right\}.\]
If ${\mathcal C}$ is not empty, then the extremal points of this convex set 
are Dirac measures or convex combinations of two Dirac measures on $\X$.
\end{lemma}

Given a real function $f$  on $\X$, for any $K\in \R$, let 
\[{\mathcal C}_K= \left\{ \nu\in {\mathcal P}(\X), \, \int e^{-(1-\alpha)f} \,d\nu=K\right\}.\]
One has 
\[\sup_ {\nu\in {\mathcal P}(\X)}\left(\int e^ {\alpha \widetilde R^\alpha f} d\nu\right)^{1/\alpha}\left(\int e^{-(1-\alpha)f} d\nu\right)^{1/(1-\alpha)}
=\sup_{K,{\mathcal C}_K\neq \emptyset } \left(\sup_{\nu\in{\mathcal C}_K} \int e^ {\alpha \widetilde R^\alpha f} d\nu\right)^{1/\alpha} K^{1/(1-\alpha)}
\]
The supremum of the  linear function  $\nu\mapsto \int e^ {\alpha \widetilde R^\alpha f} d\nu$ on the non empty  convex set ${\mathcal C}_K$ is reached at an extremal point of ${\mathcal C}_K$. Therefore, by Lemma \ref{chose}, we get 
\begin{multline*}
\sup_ {\nu\in {\mathcal P}(\X)}\left(\int e^ {\alpha \widetilde R^\alpha f} d\nu\right)^{1/\alpha}\left(\int e^{-(1-\alpha)h} d\nu\right)^{1/(1-\alpha)}
\\=\sup_ {x,y\in \X} \sup_{\lambda\in[0,1]} \left( (1-\lambda)e^ {\alpha \widetilde R^\alpha f(x)}+\lambda e^ {\alpha \widetilde R^\alpha f(y)} \right)^{1/\alpha}\left((1-\lambda) e^{-(1-\alpha)f(x)} + \lambda e^{-(1-\alpha)f(y)}\right)^{1/(1-\alpha)}
\end{multline*}
Now, let $x$ and $y$ be some fixed points of $\X$. It remains to show that for any real function $f$ on $E$ and for any $x,y\in \X$, 
\begin{eqnarray*}
\left( (1-\lambda)e^ {\alpha \widetilde R^\alpha f(x)}+\lambda e^ {\alpha\widetilde R^\alpha f(y)} \right)^{1/\alpha}\left((1-\lambda) e^{-(1-\alpha)f(x)} + \lambda e^{-(1-\alpha)f(y)}\right)^{1/(1-\alpha)}\leq 1.
\end{eqnarray*}

The left-hand side of this inequality 
is invariant by translation of the function $f$ by a constant.  Therefore, by symmetry, we  may assume that $0=f(y)\leq f(x)$. It follows that  $\widetilde R^\alpha f(y)=0$. Therefore we want to check that for any non-negative function $f$ on $\{x,y\}$,  for any $\lambda\in[0,1]$,
\[\left( (1-\lambda)e^ {\alpha\widetilde R^\alpha f(x)}+\lambda  \right)^{1/\alpha}\left((1-\lambda) e^{-(1-\alpha)f(x)} + \lambda \right)^{1/(1-\alpha)}\leq 1,\]
or equivalently, setting $\psi (\lambda)=\left((1-\lambda) e^{-(1-\alpha)f(x)} + \lambda \right)^{-\alpha/(1-\alpha)}-\lambda$,
\[ e^{\alpha \widetilde R^\alpha f(x)}\leq \inf_{\lambda\in[0,1)}\frac{\psi(\lambda)-\psi(1)}{1-\lambda}=-\psi'(1)=\frac\alpha{1-\alpha}\left(1-e^{-(1-\alpha)f(x)}\right) +1,\]
since $\psi$ is a convex function on $[0,1]$. 

So, it suffices  to check that $\widetilde R^\alpha f(x)\leq \phi(f(x))$, where
\[\phi(h)=\frac 1\alpha\log\left(\frac\alpha{1-\alpha}\left(1-e^{-(1-\alpha)h}\right)+1\right), \qquad h\geq 0.\]
The function $\phi$ is concave and $\phi(0)=0$. For all $h\geq 0$, one has 
\[\phi'(h)=\frac{1-\alpha}{e^{(1-\alpha)h}-\alpha}.\]
The function $\phi'$ is a bijection from $[0,+\infty)$ to $(0,1]$. It follows that
\[\phi(h)=\inf_{\theta\in (0,1] }\left\{\theta h + c_\alpha (1-\theta)\right\},\quad h\geq 0,\]
where  $c_\alpha$ is the  convex function defined by 
\[c_\alpha (1-\theta)= \sup_{h\in [0,+\infty)} \left\{-\theta h + \phi(h)\right\}, \quad \theta\in (0,1]. \]
After computations, we get
\[c_\alpha (u):=\frac{\alpha(1-u)\log(1-u)-(1-\alpha u)\log(1-\alpha u)}{\alpha(1-\alpha)},\]
and therefore we exactly have
 for any $x\in \X$,
\[\phi(f(x))=\inf_{\theta\in [0,1] }\left\{\theta f(x) + c_\alpha (1-\theta)\right\}=\widetilde R^\alpha f(x).\]
The proof of Lemma \ref{completegraph} is completed.
\end{proof}

\begin{proof}[Proof of Lemma \ref{chose}.] We will show that,  if $\nu\in{\mathcal C}$  is a convex combination of three probability measures $\nu_1,\nu_2,\nu_3$, 
\[\nu=\alpha_1\nu_1+\alpha_2\nu_2+\alpha_3\nu_3,\]
with $\alpha_1\neq 0$, $\alpha_2\neq 0$, $\alpha_3\neq 0$, and $\alpha_1+\alpha_2+\alpha_3=1$, and $\nu_1(\X)>0$,  $\nu_2(\X)>0$,  $\nu_3(\X)>0$, then there exists two measures $\hat \nu_1, \hat \nu_2$ in ${\mathcal C}$ and $\lambda\in [0,1] $ such  that 
\[ \nu=\lambda \hat \nu_1 + (1-\lambda )\hat \nu_2.\]

Setting $F_i= \int F d\nu_i$, for $i=1,2,3$,  we may assume, without loss of generality, that $F_1\leq F_2\leq F_3$. 
Then one has either $F_1\leq K\leq F_2$, either  $F_2\leq K\leq F_3$.  

We will assume that $F_1\leq K\leq F_2$. The case  $F_2\leq K\leq F_3$ can be treated identically and the proof in that case is let to the reader. 
Since $F_1\leq K\leq F_2$ and $F_1\leq K\leq F_3$, there exists $\beta,\gamma\in[0,1]$ such that 
\begin{eqnarray}\label{cookies}
K=\beta F_1 +(1-\beta) F_2 \quad\mbox{ and }\quad K=\gamma F_1 +(1-\gamma) F_3. 
\end{eqnarray}

If $F_1=F_3$ then $F_1=F_2=F_3=K$ and therefore $\nu_1,\nu_2,\nu_3\in {\mathcal C}$. We may choose $\lambda=\alpha_1$, $\hat \nu_1= \nu_1$ and $\hat \nu_2=\frac{\alpha_2 \nu_2+\alpha_3 \nu_3}{\alpha_2+\alpha_3}$.

If $F_1=F_2$ then necessarily $F_1=F_2=F_3=K$ and we are reduced to the previous case.

So, we may now assume that $F_1\neq F_3$ and $F_1\neq F_2$ and therefore $F_1<K\leq F_2\leq F_3$.
In that case, we  exactly have 
\[\beta= \frac{F_2-K}{F_2-F_1}\quad\mbox{ and }\quad \gamma=\frac{F_3-K}{F_3-F_1}.\]
Let us choose \[\lambda=\frac{\alpha_2}{1-\beta}=\alpha_2\,\frac{F_2-F_1}{K-F_1},\quad \hat\nu_1=\beta\nu_1+(1-\beta)\nu_2,\quad \hat\nu_2=\gamma\nu_1+(1-\gamma)\nu_2.\]
The equalities  \eqref{cookies} ensure that $\hat\nu_1\in {\mathcal C}$ and $\hat\nu_2\in {\mathcal C}$. The proof of Lemma \ref{chose} ends by checking that $\lambda \hat \nu_1+(1-\lambda)\hat \nu_2=\hat\mu$. 
One has 
\begin{eqnarray}\label{cupcake}
\qquad \lambda \hat \nu_1+(1-\lambda)\hat \nu_2=(\lambda \beta+(1-\lambda)\gamma)\nu_1+\lambda(1-\beta)\nu_2+(1-\lambda)(1-\gamma)\nu_3.
\end{eqnarray}
According to the definitions of $\lambda,\beta, \gamma$, we may easily check that $\lambda(1-\beta)=\alpha_2$,  and 
\[(1-\lambda)(1-\gamma)=\frac{K-F_1}{F_3-F_1}-\alpha_2\,\frac{F_2-F_1}{F_3-F_1}.\]
Since $\hat\mu \in {\mathcal C}$, one has $(1-(\alpha_2+\alpha_3))F_1+\alpha_2 F_2+\alpha_3 F_3$ and therefore 
\[(1-\lambda)(1-\gamma)=\alpha_3.\]
As a consequence $\lambda \beta+(1-\lambda)\gamma=1-\alpha_2-\alpha_3=\alpha_1$ and according to \eqref{cupcake}, we get
\[\lambda \hat \nu_1+(1-\lambda)\hat \nu_2=\alpha_1\nu_1+\alpha_2\nu_2+\alpha_3\nu_3=\nu.\]
\end{proof}

\vspace{0,5cm}
 \textbf{ Acknowledgements.} The author thanks R. Adamczak and D. Chafa\"i for useful discussions, that lead to consider  the Ewens distribution, and the number of cycles of given lenght, as examples  to illustrate the results of this paper.



\end{document}